\documentclass[11pt]{article}
 
\usepackage{amssymb,amsmath,amsfonts}
\usepackage{graphicx,color,enumitem}
\usepackage{mathrsfs}
\usepackage{amsthm}
\usepackage[dvipsnames]{xcolor}
\usepackage{bm}
\usepackage{comment}
\usepackage{shuffle}
\usepackage[]{appendix}
\usepackage{xcolor}
\usepackage[a4paper, total={6.5in, 9.7in}]{geometry}
\usepackage{subcaption}
\usepackage{float}
\usepackage{subcaption}
\usepackage{mathtools}

\RequirePackage[colorlinks,citecolor=blue!70!green,urlcolor=blue, linkcolor=blue!70!green]{hyperref}

\usepackage[square,numbers]{natbib}

\usepackage{tikz}
\tikzstyle{vertex}=[circle, draw, inner sep=2pt, fill=white]

\newcommand{\cc}{{\textsc{c}}}

\renewcommand{\d}{{\mathrm{d}}}

\newcommand{\E}{{\mathbb E}}

\renewcommand{\P}{{\mathbb P}}

\newcommand{\R}{{\mathbb R}}

\newcommand{\N}{{\mathbb N}}

\newcommand{\Z}{{\mathbb Z}}

\newcommand{\Fcal}{{\mathcal F}}

\newcommand{\Wcal}{{\mathcal W}}

\newtheorem{theorem}{Theorem}

\newtheorem{assump}{Assumptions}
\newtheorem{assumpB}{Assumptions}

\newtheorem{corollary}[theorem]{Corollary}

\theoremstyle{definition}
\newtheorem{definition}[theorem]{Definition}
\newtheorem{remark}[theorem]{Remark}
\newtheorem{example}[theorem]{Example}

\newtheorem{lemma}[theorem]{Lemma}

\newtheorem{proposition}[theorem]{Proposition}

\numberwithin{equation}{section}
\numberwithin{theorem}{section}

\title{{Polynomial McKean-Vlasov SDEs}}

\author{Christa Cuchiero\thanks{
Vienna University, Department of Statistics and Operations Research, Data Science Uni Vienna, Kolingasse 14-16 1, A-1090 Wien, Austria, christa.cuchiero@univie.ac.at}
\and Janka M\"oller\thanks{Vienna University, Department of Statistics and Operations Research, Kolingasse 14-16, A-1090 Wien, Austria, janka.moeller@univie.ac.at
\newline
This research was funded in whole or in part by the Austrian Science Fund (FWF) 10.55776/Y1235. \newline
The authors thank Martin Larsson and Etienne Tanré for fruitful discussions and suggestions.
}}
\date{\today}

\begin{document}

\maketitle

\begin{abstract}
We study a new class of McKean-Vlasov stochastic differential equations (SDEs), possibly with common noise,   applying the theory of time-inhomogeneous polynomial processes. 
 The drift and volatility coefficients of these SDEs depend on the state variables themselves as well as their conditional moments in a way that mimics the standard polynomial structure. Our approach leads to new results on the existence and uniqueness of solutions to such conditional McKean-Vlasov SDEs which are, to the best of our knowledge, not obtainable using standard methods.
 Moreover, we show in the case without common noise that the moments of these McKean-Vlasov SDEs can be computed by non-linear ODEs. 
  As a by-product, this also yields new results on the existence and uniqueness of global solutions to certain ODEs.
  
\end{abstract}

time-inhomogeneous Markov processes, non-linear ODEs

\tableofcontents

\section{Introduction}
Stochastic differential equations (SDEs) of McKean-Vlasov type are a class of SDEs where the drift and diffusion coefficient depend not only on time and the current state but also on the marginal law of the solution. To make this precise, let us first consider SDEs (without common noise) of the form 
\begin{align}
\begin{cases}
dX_t = B(t,X_t, \boldsymbol{\mu}_t) dt + \Sigma(t, X_t, \boldsymbol{\mu}_t) dW_t,\\
\boldsymbol{\mu}_t = \mathrm{Law}(X_t),
\end{cases}
\end{align}
where $(W_t)_{t\geq 0}$ is a one-dimensional Brownian motion. We refer to such equations as \emph{McKean-Vlasov SDEs} in the following. Of course, McKean-Vlasov SDEs can be extended to multiple dimensions, in the sense that  $X_t$  is $\mathbb{R}^d$-valued for all $t\geq0$. For the purpose of this work, we however,  restrict ourselves to the one-dimensional case. 

McKean-Vlasov equations were introduced in~\cite{Mckean} and have since then been widely studied. They gained importance in  various fields such as finance~\cite{Guyon_Book,  Zhang_phd, carmona_2015, Jourdain, Lacker_Shkolnikov}, biology~\cite{Meleard1996, population_gen1, population_gen2, chemotaxis}, physics~\cite{torus, crystal, bittencourt2013fundamentals} as well as optimal control~\cite{Guo_Ito, Pham_2017, Lacker}.

McKean-Vlasov SDEs often arise from a particle system, where the particles' dynamics depend on the empirical measure of the system and the number of particles goes to infinity. In this limit, one can then describe a "typical" particle by McKean-Vlasov dynamics, see e.g.~\cite{MV_overview, swarming} for an overview. In the context of this paper, we study the McKean-Vlasov SDEs directly, although, for applications, such SDEs supposedly arise by taking the limit of some particle system. Establishing the existence of the limit stemming from a particle system rigorously will involve specific techniques tailored to the current setting, compare e.g.,~\cite{Huber_2024, Shkolnikov}, and is beyond the scope of this paper.

In our work, we are interested in building a connection between the theory of polynomial processes and McKean-Vlasov SDEs. This allows us to develop new techniques to establish existence and uniqueness results for certain classes of McKean-Vlasov SDEs which cannot be treated by standard methods. 

Before we elaborate more on our approach, we would like to provide a brief - by no means complete - overview on existence and uniqueness results for solutions to McKean-Vlasov SDEs. The common approach consists in considering a space of probability measures $\mathscr{P}$ (often the Wasserstein-2-space) and studying the maps $B, \Sigma$ on the space $[0,\infty) \times \R^d \times \mathscr{P}$. A standard result in the literature is that existence and uniqueness of solutions to McKean-Vlasov SDEs can be established by requiring $B, \Sigma$ to be Lipschitz-continuous in all arguments, and then applying a fixed-point argument, see~\cite{Sznitman_chaos, MV_overview}. Since such Lipschitz assumptions are often too restrictive, finding milder conditions is an active field of research, and we would like to briefly point out a few such approaches. For examples, the authors of~\cite{Osgood} go beyond Lipschitz-continuity requiring instead an Osgood-type condition, however not considering any law-dependence of the diffusion term. In~\cite{Ding_Qiao}, the authors use Euler-Maruyama approximations to show existence of a solution to a martingale problem, for which they establish in turn  pathwise uniqueness leading to a strong solution. The approach taken in~\cite{Szpruch_MV} is based on Lyapunov functions to show existence and uniqueness of solutions. Another route is taken in~\cite{Mishura} where the authors study the existence of weak and strong solutions to McKean-Vlasov SDEs, however requiring non-degeneracy in the diffusion term for weak solutions and additionally Lipschitz-continuity in the state-variable of the diffusion term to construct a strong solution. Very recently,~\cite{Mao} have established existence of McKean-Vlasov SDEs under mild continuity assumptions, in particular, without requiring the dependence on the law to be Lipschitz continuous. However, in general they do not obtain uniqueness of the solution. \\

In contrast to those approaches we consider a more specific dependence of the SDEs' coefficients on the law, that is, we let them primarily depend on the marginal moments. By this restriction, we do not only get rid of Lipschitz-requirements on the diffusion term, but we avoid having to work on the space of probability measures. Instead, we can fully treat the McKean-Vlasov SDEs under our specification using the theory of time-inhomogeneous Markov processes, without explicit dependence on the marginal laws. In particular, we apply the theory of (time-inhomogeneous) polynomial processes, which was first introduced in \cite{Cuchiero_2012}. We also refer the reader to~\cite{script_Larsson} for an introduction to the topic. 

The appeal of this class of processes lies in the fact that one has an analytic expression for the (conditional) moments, which is computable via the solution of a linear ODE. For our purpose, we need to work in a time-inhomogeneous setting, which was first considered in~\cite{Hurtado_Schmidt}. The polynomial property of a Markov process is a requirement on 
infinitesimal generator. More specifically, they are assumed to map polynomials to (time-dependent) polynomials of the same or lower degree. In this context, it is therefore natural to establish also existence of solutions to SDEs via properties of their generators, that is via martingale problems. This is the approach we pursue in this work, as outlined in detail in Subsections~\ref{subsec:ex_uniq_MP} and~\ref{sec:Poly}.  \\

To be more specific, we treat, in the case without common noise, McKean-Vlasov SDEs of the following form and establish existence and uniqueness of solutions thereof: 
\begin{theorem}\label{thm:informal_poly_MV}
Consider a one-dimensional McKean-Vlasov SDE of the form
\begin{align}\label{eq:eqintro}
dZ_t &= \left(b(\E[\Z_t]) + \beta(\E[\Z_t]) Z_t \right)dt + \sqrt{c(\E[\Z_t])+ \gamma(\E[\Z_t])Z_t+\Gamma(\E[\Z_t]) Z_t^2}dW_t
\end{align}
where $(\mathbb{Z}_t)_{t\geq 0}:= (Z_t, Z^2_t, \dots, Z^N_t)_{t\geq 0}$ for $N\geq1$ and with fixed $\E[\Z_0]\in \R^N$, some maps $b,\, \beta, \,c, \,\gamma, \, \Gamma: \R^N \rightarrow \R$ and a Brownian motion $(W_t)_{t\geq 0}$. If $b, \beta, c, \gamma, \Gamma$ are locally Lipschitz continuous on $\R^N$, satisfy either the set of Assumptions~\ref{as:A} or Assumptions~\ref{as:B} (specified below), and the term under the square-root is non-negative, then there exists a unique strong solution $(Z_t)_{t\geq 0}$. Moreover, $(Z_t)_{t\geq 0}$ is a time-inhomogeneous polynomial process whose moments can be computed by a non-linear ODE, see~\eqref{eq:ode_L}.
\end{theorem}

Note that in \eqref{eq:eqintro} the polynomial structure of one-dimensional diffusions, exhibiting an affine drift and a quadratic diffusion function, is mimicked.  The respective coefficients of these functions are, however, no longer constants, but general maps of the moments. This structure is the very reason why we obtain an
autonomous non-linear ODE system for the moments.

The above result is actually obtained from two different viewpoints; a primal and a dual approach, which we will both outline in detail later. However, it is important to emphasize that even though both approaches lead to the same sufficient conditions, there are important differences between the two approaches, reflected in the fact that in some cases the existence of a solution can only be shown via one of the two approaches, see Remark \ref{rem:diff_primal_dual} for more details.

Additionally, we would like to point out that studying McKean-Vlasov SDEs from a polynomial perspective has benefits beyond establishing merely existence and uniqueness. Indeed, the solutions are in particular time-inhomogeneous polynomial processes themselves, which means that we obtain analytic expressions for their (conditional) moments. We believe that this can be of importance for modeling purposes. Moreover, we can formulate sufficient conditions for the process to remain in certain state-spaces which are subsets of $\mathbb{R}$. In particular, we can construct McKean-Vlasov SDEs such that the solution stays in a compact set, which is another aspect we believe may be useful for modeling in areas like population genetics. \\

The second part of our main contributions lies in the study of polynomial McKean-Vlasov SDEs  with a common noise term and coefficients depending on the moments conditional on the common noise. We call these processes \emph{conditional polynomial McKean-Vlasov SDEs} and present the corresponding results in Subsection~\ref{subsec:common_noise}. We establish general sufficient conditions for existence and uniqueness of solutions to this of class of McKean-Vlasov SDEs (see Theorem~\ref{thm:cond_existence_MP}, Proposition~\ref{prop:cond_existence_PMP}) and show that for a subclass the solution is even a time-homogeneous polynomial process.

In summary, this paper is structured as follows. In Section~\ref{sec:preliminaries}, we briefly introduce the theory of time-inhomogeneous polynomial processes with an emphasis on the aspects that are relevant for our main results. After that, in Subsection~\ref{subsec:ex_uniq_MP}, we first recall what it means to establish existence of a solution to an SDE via the solution to the associated martingale problem, following the work of~\cite{EK} and then extend some results of~\cite{EK} to milder growth-conditions on time-dependent coefficients of SDEs. We are then ready to present our main results on polynomial McKean-Vlasov SDEs  in Section~\ref{sec:Poly}. There, we present two approaches, called primal and dual, to establish existence of solutions.  The primal approach is then extended to the setting with common noise. 


\section{Preliminaries}\label{sec:preliminaries}

\subsection{Time-inhomogeneous polynomial processes}\label{subsec:intro_poly}
In this subsection, we introduce the notion of time-inhomogeneous polynomial processes and state some of their properties, which will be crucial for the main results of this paper.  We mainly follow~\cite{Hurtado_Schmidt}. However, while in~\cite{Hurtado_Schmidt} general time-inhomogeneous polynomial processes possibly with jumps are studied, we restrict ourselves to the case of continuous processes, 
which is the most relevant setting for us. On the other hand, the authors of~\cite{Hurtado_Schmidt} formulate everything on a finite time-horizon, while we consider a possibly infinite time-horizon. 
This extension is straightforward, and we therefore only comment briefly on the required adjustments in some of the proofs. 

Before moving to polynomial processes, we first briefly recall the definition of a Markov process as well as its associated family of transition functions and transition operators, where we mostly follow~\cite{E:16}. For further details on Markov processes, we refer to~\cite{E:16, EK}.

Throughout, we fix a filtered probability space $(\Omega, \mathcal{F}, (\mathcal{F}_t)_{t\in[0,\infty)}, \mathbb{P})$. We consider a state-space $S\subseteq \R^d$ and denote its Borel $\sigma$-algebra  by $\mathcal{S}$.

\begin{definition}[Markov Process]
 We call a stochastic process $(X_t)_{t\in [0,\infty)}$ taking values in $S$ a \emph{Markov process} if it is $(\mathcal{F}_t)_{t\in [0,\infty)}$-adapted and  
$$\mathbb{P}\left[X_t \in B\lvert \mathcal{F}_s \right]= \mathbb{P}\left[X_t \in B \lvert X_s \right] \quad \P\textrm{-a.s. for any } B \in \mathcal{S}\textrm{ and } 0\leq s \leq t < \infty.$$
\end{definition}

\begin{definition}[Transition Function]
Let us define $\Delta= \{(s,t)\, \lvert \, 0\leq s\leq t < \infty\}$. We associate with a Markov process $X$ a family of functions, called \emph{transition function}, $(p_{s,t})_{(s,t)\in \Delta}$ such that $p_{s,t}: (S, \mathcal{S}) \rightarrow [0,1]$ for all $(s,t)\in \Delta$ and $$ \P\left[X_t \in B \lvert \mathcal{F}_s\right] = p_{s,t}(X_s, B) \quad \P\textrm{-a.s. for any }(s,t)\in \Delta\textrm{ and } B\in \mathcal{S}.$$
\end{definition}

\begin{definition}[Transition Operator]\label{def:transition_op}
We can describe the Markov process $(X_t)_{t\in [0,\infty)}$ via \emph{transition operators} $(P_{s,t})_{(s,t)\in \Delta}$ with 
$$ P_{s,t} f(x) = \int_S f(\xi) p_{s,t}(x, d\xi)$$
for all $x\in S$ and $\mathcal{S}$-measurable functions $f:S\rightarrow \R_{\geq 0}$. Moreover, these transition operators satisfy the \emph{Chapman-Kolmogorov} equations 
\begin{equation*}
P_{s,u} f(x)= P_{s,t}P_{t,u} f(x) 
\end{equation*}
for all $0\leq s \leq t \leq u < \infty$ and $f$ as above. 
\end{definition}

Following~\cite{Hurtado_Schmidt}, time-inhomogeneous polynomial processes are  a particular type of Markov process, for which we shall need the subsequent notions.

We denote by $\mathbf{n}:= (n_1, \dots, n_d)\in \{0, 1, \dots\}^d$ a multiindex, where $\lvert \mathbf{n}\rvert := n_1+ \dots + n_d$ and for  $x\in \R^d$,  $x^{\mathbf{n}}:= x_1^{n_1}\cdots x_d^{n_d}$. Let us introduce the following spaces of (time-dependent) polynomials up to degree $m\geq 0$: 
\begin{align*}
\mathcal{P}_m(S) &:= \left\{x \mapsto \sum_{\lvert \mathbf{n}\rvert=0}^m a_{\mathbf{n}} x^{\mathbf{n}}\, \lvert \, x \in S, a_{\mathbf{n}}\in \R \right\}\\
\tilde{\mathcal{P}}_m(S)&:= \left\{ (s,t,x) \mapsto \sum_{\lvert \mathbf{n}\rvert=0^m} a_{\mathbf{n}}(s,t)x^{\mathbf{n}} \, \lvert \,  (s,t,x) \in \Delta\times S, \, a_{\mathbf{n}}(\cdot) \in C^1(\Delta, \R) \right\}.
\end{align*}

\begin{definition}[Polynomial Process]
We call a family of transition operators $(P_{s,t})_{(s,t)\in \Delta}$ \emph{polynomial} if for all $m\geq 0$ and $f\in \mathcal{P}_m(S)$ it holds that  
\begin{equation*}
\left((s,t,x) \mapsto P_{s,t}f(x)\right) \in \tilde{\mathcal{P}}_m(S).
\end{equation*}
Moreover, we call a Markov process $(X_t)_{t\in [0,\infty)}$ a \emph{$($time-inhomogeneous$)$ polynomial process} if its associated transition operators are polynomial. 
\end{definition}

\begin{remark}
Let us note that in the above definition, we have implicitly assumed that polynomials are $\mathcal{S}$-measurable and, more crucially, that all moments exist.
\end{remark}

\begin{definition}[Infinitesimal Generator]\label{def:infi_generator}
	For a Markov process with transition operators $(P_{s,t})_{(s,t)\in\Delta}$ we define the family of \emph{infinitesimal generators} $(\mathcal{H}_s)_{s\in [0,\infty)}$ as the linear operators given by,
	\begin{equation*}
		(\mathcal{H}_sf)(x) := \lim_{h\searrow 0} \frac{(P_{s, s+h} f- f)(x)}{h} \quad \textrm{for all } x\in S, 
	\end{equation*}
	whenever this limit exists. We denote the set of $\mathcal{S}$-measurable functions $f:S \rightarrow \R$ for which the limit exists by $\mathcal{D}(\mathcal{H}_s)$.

\end{definition}

\begin{proposition}
	Let $(\mathcal{H}_s)_{s\in [0, \infty)}$ be the infinitesimal generator associated with a polynomial transition operator, then it holds for all $m\geq 0$ and $s\in [0,\infty)$ that $\mathcal{P}_m(S)\subseteq \mathcal{D}(\mathcal{H}_s)$ and $\mathcal{H}_s(\mathcal{P}_m(S)) \subseteq \mathcal{P}_m(S)$.
\end{proposition}

\begin{proof}
	The result follows from the fact that a polynomial transition operator maps polynomials to polynomials with continuously differentiable time-dependent coefficients. Hence, the limit translates to taking the derivative of these coefficients. Note that the coefficients of $\mathcal{H}_sf$ for $f \in \mathcal{P}_m(S)$ depend on $s \in [0,\infty)$. See~\cite[Proof of Lemma 3.3]{Hurtado_Schmidt} for further details. 
\end{proof}

Having introduced the most important concepts of time-inhomogeneous polynomial processes, we  focus now on 
\emph{time-inhomogeneous polynomial diffusions,}
 which will be important for the main results of this paper.

\begin{definition}[Polynomial Diffusion]\label{def:poly_diffusion}
	Let $(X_t)_{t\in [0,\infty)}$ be a continuous  semimartingale with unique decomposition $X_t= X_0 +A_t +M_t$ where $(A_t)_{t\in [0,\infty)}$ is a process of finite variation and $(M_t)_{t\in[0,\infty)}$ is a local martingale. We denote by $(\mathfrak{b},\mathfrak{c})$ the differential characteristics of $X$, that is, 
	\begin{align*}
		A^i &= \int_0^{\cdot} \mathfrak{b}^i(s, X_s) ds \\
		[X^i, X^j]_{\cdot} &= \int_0^{\cdot} \mathfrak{c}^{i,j}(s, X_s) ds
	\end{align*}
	where $[\cdot, \cdot]$ denotes the quadratic variation. We then call $(X_t)_{t\in [0,\infty)}$ a \emph{polynomial diffusion} if for all $i, j \in \{1, \dots, d\}$, $t \in [0,\infty)$, $x\in S$ it holds that 
	\begin{enumerate}
		\item $\mathfrak{b}^i(t, \cdot) \in \mathcal{P}_1(S)$ and $\mathfrak{b}^i(\cdot, x)\in C([0,\infty))$, and,
		\item $\mathfrak{c}^{i,j}(t, \cdot) \in \mathcal{P}_2(S)$ and $\mathfrak{c}^{ij}(\cdot , x)\in C([0,\infty))$.
	\end{enumerate}	
\end{definition}

\begin{proposition}[{\cite[Corollary 5.5]{Hurtado_Schmidt}}] \label{prop:poly_diffusion}
	A polynomial diffusion $(X_t)_{t\in[0,\infty)}$ is a polynomial process with infinitesimal generator given by 
	\begin{equation*}
		(\mathcal{H}_sf)(x)= \sum_{i=1}^d D_i f(x)\mathfrak{b}^i(s,x) + \frac{1}{2} \sum_{i,j=1}^d D_{i,j}f(x)\mathfrak{c}^{i,j}(s,x)
	\end{equation*}
	for all $f\in \mathcal{P}_m(S)$, $m\geq 0$, $s\in [0, \infty)$, $x\in S$.
\end{proposition}

\begin{proof}
The proof of this hinges on~\cite[Theorem 3.7, Theorem 5.4 and Lemma 5.4]{Hurtado_Schmidt} and the extension for these results to hold for all $t\in [0,\infty)$ deserves a brief discussion. First of all, regarding~\cite[Lemma 5.4]{Hurtado_Schmidt}, the coefficients $\tilde{C}_1, \tilde{C}_2$ now depend on time, but they remain finite over all of $[0,\infty)$ and hence~\cite[Lemma 5.4]{Hurtado_Schmidt} holds analogously on $[0, \infty)$. Moreover, Theorem~\cite[Lemma 5.3]{Hurtado_Schmidt} holds on $[0,\infty)$ by noting that the Kolmogorov forward and backward equations hold analogously. For the forward equation, see~\cite{Rueschendorf2016}. A solution for the backward equation can be constructed exactly as described in the proof of~\cite[Theorem 3.7]{Hurtado_Schmidt}; knowing a solution to the forward equation exists for all $t\in [s, \infty)$, we can apply~\cite[Lemma A.3]{Hurtado_Schmidt} and obtain a solution to the backward equation for all $s\in [0,t)$ where $t$ is arbitrary. 
\end{proof}

Let us also point out that a diffusion as defined above can always be realized (in law) via an SDE, see~\cite{Jacod_Shiryaev}. The canonical one-dimensional example of an SDE in this setting is the following:

\begin{example}\label{ex:poly_SDE}
Let us consider an SDE of the form: 
$$ dX_t = (b(t)+ \beta(t)X_t) dt + \sqrt{c(t) + \gamma(t)X_t + \Gamma(t)X^2_t}dW_t,$$
where $X_0=x$, the maps $b, \beta, c, \gamma, \Gamma: [0,\infty) \rightarrow \R$ are continuous and $(W_t)_{t\geq0}$ denotes a one-dimensional Brownian motion. Moreover, assume the term under the square-root is non-negative, $\beta, \Gamma$ are uniformly bounded and there exists a continuous and strictly increasing map $h: [0, \infty ) \rightarrow [0, \infty)$ such that $b, c ,\gamma$ are all dominated by $h$. Then there exists a path-wise unique strong solution $(X_t)_{t\geq 0}$ to this SDE, by Theorem~\ref{thm:sol_MP} and Proposition~\ref{prop:pmp}. Moreover, looking at the differential characteristics
\begin{align*}
\mathfrak{b}(t, x) &= b(t)+ \beta(t)x \\
\mathfrak{c}(t, x) &= c(t) + \gamma(t)x + \Gamma(t)x^2
\end{align*}
$(X_t)_{t\geq0}$ is a time-inhomogeneous polynomial process, in particular a polynomial diffusion, by Definition~\ref{def:poly_diffusion}.
\end{example}

An important property of time-inhomogeneous polynomial processes is that we have an analytic expression for the conditional moments of the process. Before we state this so-called \emph{moment-formula}, let us make explicit that the transition operators and infinitesimal generators of polynomial processes act \emph{linearly} on polynomials and therefore admit a matrix representation with respect to a polynomial basis. 

To this end, let us consider polynomials up to an arbitrary but fixed degree $m>0$. Moreover, we denote by $(v_0, \dots, v_N)$  a basis such that $\mathcal{P}_m(S) = \textrm{span}\{ (v_0, \dots, v_{N(m)} \}$, where $N(m)$ is the dimension of the degree-$m$ polynomials on $S$. For a linear operator $A: \mathcal{P}_m(S) \rightarrow \mathcal{P}_m(S)$ we denote its matrix representation by $(\mathbf{A}_{ij})_{0\leq i,j \leq N(m)}$ such that for all $i \in \{0, \dots, N(m)\}$ 
\begin{equation}\label{eq:transform_operator}
    Av_i = \sum_{j=0}^{N(m)} \mathbf{A}_{ji} v_j.
\end{equation}

Indeed, for polynomial processes and for fixed time-points $s \in [0,\infty)$ (resp.~$(s,t)\in \Delta$) the operators $\mathcal{H}_s$ (resp.~$P_{s,t}$) act linearly on the space of polynomials. Therefore, we associate with $\mathcal{H}_s$ a matrix $\mathbf{H}_s$ for each $s\in [0,\infty)$ and likewise with $P_{s,t}$ a matrix $\mathbf{P}_{s,t}$ for each $(s,t) \in \Delta$. We do want to emphasize that these matrices are time-dependent. Let the following example illustrate the structure of the matrix $\mathbf{P}_{s,t}$:

\begin{example}\label{ex:P_op}
Let us denote by $(\mathbf{P}_{s,t})_{:, i}:= \mathbf{P}_{s,t}\cdot e_i$ the $i$-th column of $\mathbf{P}_{s,t}$. Moreover, let us consider $S=\R$ and choose the canonical basis $v_i= (x \mapsto x^i)$ for all $i\in \{0, \dots, N\}$. 
Then, by definition, it holds that 
\begin{equation*}
\E[X_t^i\lvert X_s=x] = \E[v_i(X_t) \lvert X_s=x] = (1, x, \dots, x^N)\cdot (\mathbf{P}_{s,t})_{:, i}.
\end{equation*}
Hence, the columns of $\mathbf{P}_{s,t}$ contain the \emph{coefficients} of the polynomials to obtain the $i$-th moment. 
\end{example}

We are now ready to state the moment formula: 
\begin{theorem}[Moment Formula]\label{thm:moment_formula}
	For a polynomial process $(X_t)_{t\in [0,T]}$ it holds for any $f\in \mathcal{P}_m(S)$ where $f= \sum_{i=0}^{N(m)} u_i v_i$ that 
	\begin{equation*}
		\E[f(X_t)  \lvert \mathcal{F}_s] = \sum_{i,j=0}^{N(m)} v_i(X_s) \big(e^{\mathbf{\Omega}(s,t)}\big)_{ij} u_j
	\end{equation*}
	for $u\in\R^{N(m)}$ and where $\mathbf{\Omega}(s,t)= \sum_{k=0}^{\infty} \mathbf{\Omega}_k(s,t)$ is given by the Magnus expansion of $(\mathbf{H}_s)_{s\in[0,T]}$ assuming we choose a basis such that $\int_0^T \| \mathbf{H}_s\|_2 ds < \pi$ for the spectral norm $\| \cdot \|_2$.\footnote{Note that the spectral norm depends on the choice of the basis. Let a matrix $\mathbf{A}$ have spectral norm $\alpha:=\| \mathbf{A}\|_2$ under basis $(v_0, \dots, v_N)$, then it has spectral norm $\epsilon \alpha$ under the re-scaled basis $(\epsilon^{-1}v_0, \dots, \epsilon^{-1}v_N)$ for any $\epsilon>0$. See~\cite[Section 7 \& 9]{Hurtado_Schmidt} for more details.} In particular - denoting by $[\cdot, \cdot ]$ the commutator of two matrices, i.e. $[\mathbf{A}, \mathbf{B}]= \mathbf{A}\mathbf{B} - \mathbf{B}\mathbf{A}$ - the first three elements of the Magnus expansion are given by
	$$ \mathbf{\Omega}_1(s,t) = \int_s^t \mathbf{H}_r dr $$
	$$ \mathbf{\Omega}_2(s,t) = -\frac{1}{2} \int_s^t \int_s^{t_1} [\mathbf{H}_{t_1}, \mathbf{H}_{t_2}]   dt_2 dt_1$$
	$$ \mathbf{\Omega}_3(t) = \frac{1}{6} \int_s^t \int_s^{t_1} \int_s^{t_2} [\mathbf{H}_{t_1}, [\mathbf{H}_{t_2}, \mathbf{H}_{t_3}]] + [\mathbf{H}_{t_3}, [\mathbf{H}_{t_2}, \mathbf{H}_{t_1}]] dt_3 dt_2 dt_1 .$$
	Hence, if $[\mathbf{H}_s, \mathbf{H}_t]=0$ for all $s,t\geq 0$, then the above simplifies to 
	\begin{equation*}
		\E[f(X_t)  \lvert \mathcal{F}_s] = \sum_{i,j=0}^{N(m)} v_i(X_s)\big(e^{\int_s^t\mathbf{H}_rdr}\big)_{ij}\cdot u_j.
	\end{equation*}
	
\end{theorem}

\begin{proof}
	Using Lemma~\ref{lem:intro_fwd_bwd} below, the result follows from standard ODE theory. We refer the reader to~\cite[Section 7]{Hurtado_Schmidt} for more details.
\end{proof}

\begin{remark}
In the time-homogeneous case, the infinitesimal generator in constant in time, i.e. $\mathbf{H}_s \equiv \mathbf{H}$ for all $s\geq 0$. Hence, the moment formula simplifies to  \begin{equation*}
		\E[f(X_t)  \lvert \mathcal{F}_s] = \sum_{i,j=0}^{N(m)} v_i(X_s)\big(e^{(t-s)\mathbf{H}}\big)_{ij}\cdot u_j.
	\end{equation*}

\end{remark}

\begin{lemma}\label{lem:intro_fwd_bwd}
	For a polynomial process $(X_t)_{t\in [0,T]}$ its transition operator and infinitesimal generator are described by the Kolmogorov forward and backward ODEs. We denote by $\frac{d^+}{dt}$ the \emph{right-hand-derivative}, that is $\frac{d^+}{dt} f(t) := \lim_{h\searrow 0} \frac{f(t+h)- f(t)}{h}$ .
	\begin{enumerate}
		\item for every fixed $t\in [0,\infty)$ it holds on $[0, t)$ that $$\frac{d^+}{ds}\mathbf{P}_{s,t} = -\mathbf{H}_s\mathbf{P}_{s,t} \textrm{ with } \mathbf{P}_{t,t}= id$$
		\item for every fixed $s\in [0,\infty)$ it holds on $[s, \infty)$ that $$\frac{d^+}{dt}\mathbf{P}_{s,t} = \mathbf{P}_{s,t}\mathbf{H}_t \textrm{ with } \mathbf{P}_{s,s}= id.$$
	\end{enumerate}
\end{lemma}

\begin{proof}
	This follows essentially from the definition of the infinitesimal generator and the Chapman-Kolmogorov equation. See~\cite[Propostion 7.1]{Hurtado_Schmidt} and~\cite{Rueschendorf2016}. See also our comments in the proof of Proposition~\ref{prop:poly_diffusion} regarding the extension to $[0, \infty)$.
\end{proof}

\subsection{Existence and uniqueness of solutions to time-inhomogeneous diffusions}\label{subsec:ex_uniq_MP}

In this section, we present the approach to establish existence of weak solutions to SDEs via existence of solutions to the \emph{martingale problem}. First, we briefly outline the concept behind this approach and refer the reader to~\cite[Chapter 4 and Chapter 5.3]{EK} for more details. In Proposition~\ref{prop:pmp} and Theorem~\ref{thm:sol_MP} we present sufficient conditions on the coefficients of the SDE for a solution to the martingale problem to exist in a time-inhomogeneous setting. While Proposition~\ref{prop:pmp} follows the same reasoning as in the time-homogeneous setting outlined e.g., in~\cite{script_Larsson}, Theorem~\ref{thm:sol_MP} is an extension of~\cite[Proposition 5.3.5 and Theorem 5.3.10]{EK} allowing for milder growth-conditions of the time-dependencies. Let us start  by the definition of a solution to the martingale problem (see~\cite[Chapter 4.3]{EK} for more details). 

\begin{definition}[Solution to the Martingale Problem]
Let $(X_t)_{t\geq 0}$ be a stochastic process with state-space $S\subseteq \R^d$ on some filtered probability space and let $A: \mathcal{D}(A)\subseteq C_b(S) \rightarrow C_b(S)$ 
be a linear operator for $C_b(\dot)$ denoting the set of bound continuous functions. We call $(X_t)_{t\geq 0}$ a \emph{solution to the martingale problem for} $A$ if 
$$ f(X_t) - \int_0^t (Af)(X_s)ds$$
is a martingale for all $f\in \mathcal{D}(A)$.
\end{definition}

For our purpose, we are interested in the solution to martingale problems for a given (possibly time-extended) \emph{generator of an SDE} $\mathcal{H}^0$ with $\mathcal{D}(\mathcal{H}^0)= C^{\infty}_c(S) $ i.e. the smooth functions of compact support on the state space $S$.

Let us from now on consider a one-dimensional SDE of the form 
\begin{equation}
dX_t = B(t, X_t) dt + \Sigma(t, X_t) dW_t
\end{equation}
for continuous functions $B, \sigma: [0,\infty) \times S \rightarrow \R$, $S\subseteq \R$, some given initial distribution for $X_0$ and $(W_t)_{t\geq0}$ denoting a one-dimensional Brownian motion. We can associate with this SDE a family of \emph{time-dependent} generators $(\mathcal{H}_s)_{s\geq 0}$ acting on $f\in C_c^{\infty}(S)$ as 
$$(\mathcal{H}_sf)(x) = B(s, x)f'(x) + \frac{1}{2} \Sigma^2(s,x) f''(x)$$
for all $s\geq 0$ and $x\in S$. We refer the reader to~\cite[Propsition 5.3.1 and Lemma 5.3.2]{EK} on how to construct a weak solution of an SDE given a solution to the martingale problem for the generator of the SDE. For brevity, we merely note here that a weak solution to an SDE can be constructed, if a solution to the martingale problem of its generator exists.

Let us point out, that in the time-inhomogeneous case, we need to study the time-extended martingale problem to establish existence of a solution, see~\cite[Theorem 4.7.1 and Proposition 5.3.5]{EK}. For an SDE with time-dependent coefficients and associated generators $(\mathcal{H}_s)_{s\geq0}$ we define the generator of the time-extended process $(t, X_t)_{t\geq0}$ as 
\begin{align*}
    \mathcal{H}^0 : \{\tau f \,  \lvert \, (\tau, f) \in C_c^1([0,\infty)) \times C_c^{\infty}(S)\} &\rightarrow C([0,\infty) \times S)\\
    \tau f &\mapsto \mathcal{H}^0\tau f
\end{align*}
such that 
\begin{align*}
\mathcal{H}^0\tau f : [0,\infty) \times S &\rightarrow \R\\
(t,x) &\mapsto \tau(t) \mathcal{H}_t f(x) + \tau'(t) f(x).
\end{align*}
As in \cite[Proposition 5.3.5]{EK} we study the existence of the solution fully in the time-extended setting first and then project the solution down to the space-component. 

Establishing the existence of a solution $(t, X_t)_{t\geq0}$ to the martingale problem for $\mathcal{H}^0$ can be broken down to the following two essential steps (see~\cite[Proposition 5.3.5 and Theorem 5.3.10]{EK}):
\begin{enumerate}
    \item showing that $\mathcal{H}^0$ satisfies the \emph{positive maximum principle}\footnote{We will introduce the positive maximum principle in Definition~\ref{def:pmp}.} (see~\cite[Theorem 4.5.4]{EK}) 
    \item verifying that the solution $(t,X_t)_{t\geq 0}$ has continuous sample paths with values in $[0, \infty)\times S$ (see~\cite[Proposition 3.9]{EK})
\end{enumerate}

Clearly, both of the above steps will impose some conditions on the generator $\mathcal{H}^0$ which translate to conditions on the SDE's coefficients $B$ and $\sigma$. Let now us tackle step (i):

\begin{definition}[Positive Maximum Principle, {see~\cite[p. 165]{EK}}]\label{def:pmp}
We say a linear operator $A$ satisfies the \emph{positive maximum principle (PMP)} on $S$ if for $x_0\in S$ such that $f(x_0)= \sup_{x\in S} f(x)$ for $f\in \mathcal{D}(A)$ it holds that $(Af)(x_0)\leq 0$.
\end{definition}

We now establish conditions on the coefficients $B, \Sigma$ for the positive maximum principle to be fulfilled on a given state-space $S\subseteq \R$ in the next proposition. Although we consider the time-inhomogeneous setting here, the reasoning is very similar to the time-homogeneous case outlined in~\cite[Examples 4.15, 4.16 and 4.18]{script_Larsson}. 

\begin{proposition}\label{prop:pmp}
Let 
$$dX_t = B(t, X_t) dt + \Sigma(t, X_t) dW_t$$
and denote by $\mathcal{H}^0$ the associated time-extended generator.
\begin{enumerate}
\item For $S=\R$, the generator $\mathcal{H}^0$ satisfies the PMP without further assumptions as $\Sigma^2(\cdot, \cdot)\geq 0$ on all of $[0,\infty) \times \R$. 

\item For $S=\R_{\geq 0}$, $\mathcal{H}^0$ satisfies the PMP, if $\Sigma^2(\cdot, 0)\equiv 0$ and $B(\cdot, 0)\geq 0$ for $x=0$

\item For $S=[0,1]$, $\mathcal{H}^0$ satisfies the PMP,  if $\Sigma^2(\cdot, 0) \equiv \Sigma^2(\cdot, 1) \equiv 0$ and the drift term fulfills 
$B(\cdot, 0)\geq 0$, $B(\cdot, 1)\leq 0$. 
\end{enumerate}
\end{proposition}

\begin{proof}
Let $(t_0,x_0) \in [0, \infty) \times S$ such that $\tau(t_0)f(x_0)= \sup_{(t, x) \in [0, \infty) \times S} \tau(t)f(x)$. We note that $\tau(t_0)f(x_0)\geq 0$ since the functions have compact support on $[0, \infty) \times S$. Denoting by $S^{\circ}$ the interior of $S$ it clearly holds that for $(t_0, x_0) \in (0, \infty) \times S^{\circ}$
$$ \tau'(t_0)f(x_0)=0, \quad \tau(t_0)f'(x_0) =0, \quad \tau(t_0)f''(x_0) \leq 0.$$ 
Hence, since $\Sigma^2(\cdot, \cdot) \geq 0$ it is straightforward to check that $\mathcal{H}^0$ satisfies the PMP on $(0,\infty) \times S^{\circ}$. Moreover, if $t_0=0$, it holds that $$\tau'(t_0)f(x_0) \leq 0, \quad \tau(t_0)f'(x_0) =0, \quad \tau(t_0)f''(x_0) \leq 0$$
and it is straightforward to check that $\mathcal{H}^0$ therefore fulfills the PMP on $[0, \infty) \times S^{\circ}$. We are now left to check the cases where $S$ contains some boundary points, i.e.~$\partial S \cap S \neq \emptyset$. This case is relevant for $S=\R_{\geq 0}$ and $S=[0,1]$ and yields the additional conditions. If $x_0 \in \partial S \cap S$ and $t_0 \in [0, \infty)$, it holds that 
$$ \begin{cases}
    \tau(t_0)f'(x_0) \leq 0 \quad &\textrm{if } x_0 = \inf{x\in S}\\
    \tau(t_0)f'(x_0) \geq 0 \quad &\textrm{if } x_0 = \sup{x\in S}.\\
\end{cases}$$
Note that we can not say anything about the second order derivative of $f$ at the boundary and therefore require $\Sigma^2$ to vanish on boundary points of $S$. It is easy to check that the conditions on $B$ at the boundary points are exactly what we need for the PMP to be satisfied. 
\end{proof}

Having taken care of (i), we address step (ii) in the following theorem: 

\begin{theorem}\label{thm:sol_MP}
Consider a time-inhomogeneous SDE of the form
$$dX_t = B(t, X_t) dt + \Sigma(t, X_t) dW_t,$$
with state-space $S=\R$, $S= \R_{\geq 0}$ or $S=[0,1]$ and  $B, \Sigma: [0,\infty) \times S \rightarrow \R$ continuous such that the infinitesimal generator $\mathcal{H}^0$ satisfies the positive maximum principle on the corresponding state space. Then there exists a  solution to the corresponding martingale problem, if there is a map $h(t):[0, \infty) \rightarrow [0,\infty)$ which is continuous and strictly increasing s.t. 
\begin{equation}\label{eq:sigma_growth}
\Sigma^2(t,x) \leq K^{\Sigma}(1+ h^2(t)+ h(t)x+ x^2) 
\end{equation}
and 
\begin{equation}\label{eq:xb_growth}
x\cdot B(t,x) \leq K^{B}(1+ h(t)x+ x^2)
\end{equation}
for all $x\in S$ and $t\in [0, \infty)$ and some constants $K^{\Sigma}, K^{B}\geq 0$. Moreover, this solution is path-wise unique if $\Sigma^2(t,x), B(t,x)$ are locally Lipschitz in $x$ uniformly in $t$ on bounded intervals, that is: for any open subset $U \subset S$ and $T>0$ there is a constant $K$ such that 
$$\lvert \Sigma^2(t, x)- \Sigma^2(t,y) \rvert \vee \lvert B(t,x) - B(t,y) \rvert \leq K \lvert x-y\rvert$$
for all $x,y\in U$ and $t\in [0,T]$.
\end{theorem}

\section{Polynomial McKean-Vlasov equations}\label{sec:Poly}

In this section, we study sufficient conditions for existence and uniqueness of solutions to one-dimensional McKean-Vlasov SDEs. In particular, we are concerned with studying McKean-Vlasov SDEs which are of the following form:

\begin{definition}[Polynomial McKean-Vlasov SDEs]
Let $(W_t)_{t\in [0,\infty)}$ be a one-dimensional Brownian motion. We call an SDE a \emph{polynomial McKean-Vlasov SDE} if it is given by
\begin{align}\label{eq:poly_MV_SDE}
dZ_t &= \left(b(\E[Z_t],\dots, \E[Z_t^N]) + \beta(\E[Z_t],\dots, \E[Z_t^N])Z_t \right)dt \\ \nonumber &\quad+ \sqrt{c(\E[Z_t],\dots, \E[Z_t^N])+ \gamma(\E[Z_t],\dots, \E[Z_t^N])Z_t+\Gamma(\E[Z_t],\dots, \E[Z_t^N])Z_t^2}dW_t
\end{align}
for some $N\geq1$, continuous maps $b,\,\beta, \,c, \,\gamma, \,\Gamma: \R^N \rightarrow \R$ and $Z_0$ with $\E[Z_0], \dots, \E[Z_0^N] < \infty$. 
\end{definition}

For now the SDE in~\eqref{eq:poly_MV_SDE} is only formal, since we have not established the existence of solutions, which is the main goal of this section. We shall distinguish  several state-spaces for $(Z_t)_{t\geq 0}$ which will lead to certain restrictions on the  maps $b, \beta, c , \gamma, \Gamma$. One obvious condition is that the term under the square-root must be positive, as we want the solution to be real-valued. We discuss possible state spaces in more detail in Theorem~\ref{thm:poly_MV}, Remark~\ref{rem:poly_MV_S} and Corollary~\ref{coro:poly_MV_S}. We would like to point out that certain SDEs that are contained in the class of~\ref{eq:poly_MV_SDE} have already been studied in the 1980s, see~\cite{periodic_law}, in the context of diffusions with a period law. Note that the form of \eqref{eq:poly_MV_SDE} strongly resembles the one of the SDE in Example~\ref{ex:poly_SDE}. This resemblance is by no means a coincidence. In fact, we want to solve SDEs of the form of~\eqref{eq:poly_MV_SDE} using the theory of time-inhomogeneous polynomial processes. To do so, we can follow two approaches, a \emph{primal} and a \emph{dual} one, which we explore in detail in Section~\ref{subsec:primal} and~\ref{subsec:dual_poly} respectively. 

The primal approach is rather straightforward: we apply Itô's lemma to find a system of ODEs describing the moments $(\E[Z_t], \dots, \E[Z^N_t])_{t\geq 0}$. We then establish conditions under which we obtain a unique global solution to the corresponding system of ODEs and hence have a full description of the (deterministic) process $(\E[Z_t], \dots, \E[Z^N_t])_{t\geq 0}$. Plugging this process back into the original SDE~\eqref{eq:poly_MV_SDE}, we establish existence and uniqueness of solutions to the SDE via a solution of the associated martingale problem and conclude that this solution is a time-inhomogeneous polynomial process. 

The dual approach is slightly more involved: we start by considering a time-inhomogeneous polynomial process, depending on a deterministic auxiliary process. At first glance, this has nothing to do with the McKean-Vlasov SDEs. However, the auxiliary process is required to satisfy a system of ODEs, which are in fact associated with the Kolmogorov forward and backward equations of the process' transition operator. Using the moment formula, we can then conclude that the auxiliary process coincides with $(\E[Z_t], \dots, \E[Z^N_t])_{t\geq 0}$. 

As a third consideration of this section, we extend the primal approach to a common-noise setting, where the moments are replaced by the conditional moments (conditional on the common noise). We present these results in Section~\ref{subsec:common_noise}.

\subsection{Primal point of view}\label{subsec:primal}
Recall the form of the polynomial McKean-Vlasov SDEs given by~\eqref{eq:poly_MV_SDE}. To shorten the notation, we write $\Z_t := (Z_t, \dots, Z_t^N)$ and accordingly $\E[\Z_t] := 
(\E[Z_t], \dots, \E[Z_t^N])$. Using Itô's formula, it is then straightforward to derive for $k\in \{2, \ldots, N\}$
\begin{align}
dZ_t^k &= \Big[ k\cdot\big(b(\E[\Z_t])+ \beta(\E[\Z_t])Z_t\big) Z_t^{k-1}\\& \qquad + \frac{k(k-1)}{2}\big(c(\E[\Z_t])+ \gamma(\E[\Z_t])Z_t + \Gamma(\E[\Z_t])Z_t^2\big)Z_t^{k-2} \Big] dt \\& \quad + k \cdot \sqrt{c(\E[\Z_t])+ \gamma(\E[\Z_t])Z_t + \Gamma(\E[\Z_t])Z_t^2}Z_t^{k-1} dW_t. 
\end{align}
We would then like to take the expectation and obtain a non-linear ODE for the moments. However, at this stage we do not know whether $$M^{(k)}_t:= \int_0^{t} \sqrt{c(\E[\Z_s])+ \gamma(\E[\Z_s])Z_s + \Gamma(\E[\Z_s])Z_s^2}Z_s^{k-1} dW_s$$ is a true martingale or just a local martingale. Note that $(M^{(k)}_t)_{t\geq 0}$ is a true martingale if 
\begin{equation} \label{eq:true_mart}
\int_0^{t} \E\Big[\Big\lvert \big(c(\E[\Z_s])+ \gamma(\E[\Z_s])Z_s + \Gamma(\E[\Z_s])Z_s^2\big) Z_s^{2k-2}\Big\rvert\Big] ds< \infty
\end{equation}
for $t\geq 0$.
To establish this for all $k \in \{1, \ldots, N\}$, we impose certain growth-conditions on the maps $b, \beta, c, \gamma, \Gamma$. We present two sets of assumptions on these maps under which we can establish the true martingale property and later the existence of a solution to~\eqref{eq:poly_MV_SDE}.

\begin{definition}\label{def:norm_N}
Let us define $\|x\|_N:= (\sum_{i=1}^N |x_i|^{\frac{2}{i}})^{\frac{1}{2}}$ for any $x\in \R^N$. 
\end{definition}

\begin{assump} \label{as:A}
Let the following hold for all $x = (x_1, \dots, x_N)^{\mathsf{T}} \in \R^N$ 
\begin{enumerate}
\item $\beta: \R^N \rightarrow [-\beta_0, \beta_0]$ and $\Gamma: \R^N \rightarrow [-\Gamma_0, \Gamma_0]$ for some $\beta_0, \Gamma_0\geq 0$;
\item there is a $c_0\in \R$ and a continuous increasing function $f_c: \R_{\geq 0} \rightarrow \R_{\geq 0}$ such that $|c(x)| \leq c_0(1 + f_c(\lvert x_1\rvert)+  \|x\|_N^2)$;
\item there are $b_0,\gamma_0 \in \R$ and a continuous increasing function $f_{\gamma}: \R_{\geq 0} \rightarrow \R_{\geq 0}$ such that  $\lvert \gamma(x)\rvert \leq \gamma_0 (1+ f_{\gamma}(\lvert x_1\rvert) + \|x\|_N)$ and $\lvert b(x) \rvert \leq  b_0(1+ \lvert x_1\rvert )$.
\end{enumerate}
\end{assump}

\begin{assump}\label{as:B}
We assume that for all $x = (x_1, \dots, x_N)^{\mathsf{T}} \in \R^N$ 
\begin{enumerate}
\item $\beta: \R^N \rightarrow [-\beta_0, \beta_0]$ and $\Gamma: \R^N \rightarrow [-\Gamma_0, \Gamma_0]$ for some $\beta_0, \Gamma_0\geq 0$;
\item there is a $c_0\in \R$  such that $|c(x)| \leq c_0(1 +  \|x\|_N^2)$;
\item there are $ b_0,\gamma_0,\in \R$ such that  $\lvert \gamma(x)\rvert \leq \gamma_0 (1 + \|x\|_N)$ and $\lvert b(x) \rvert \leq  b_0(1+ \| x\|_N )$.
\end{enumerate}
\end{assump}

\begin{lemma}\label{lem:new_ass_maps}
Assume that a solution $(Z_t)_{t\geq 0}$ to~\eqref{eq:poly_MV_SDE} exists and that either Assumptions~\ref{as:A} \emph{or} Assumptions~\ref{as:B} hold. Then $(M_t)_{t\geq 0}$ is a true martingale. 
\end{lemma}

We refer the reader to Appendix~\ref{app:new_ass_maps} for the proof of Lemma~\ref{lem:new_ass_maps}.

\begin{example}\label{ex:maps_primal}
Let us present some examples for the maps $c, b, \gamma: \R^N \rightarrow \R$ which fulfill the set of Assumptions~\ref{as:A}:
\begin{enumerate}
    \item $c(x)= c_0 + \sum_{k=1}^N c_k\lvert x_k\rvert^{\frac{2}{k}} + c_{N+1} \exp(x_1)$ for $c_0, \dots, c_{N+1} \in \R$;
    \item $b(x)= b_0 + b_1\lvert x_1\rvert$ for $b_0, b_1 \in \R$;
    \item $\gamma(x)= \gamma_0+ \sum_{k=1}^N \gamma_k\lvert x_k \rvert^{\frac{1}{k}}+ \gamma_{N+1}\exp(x_1)$ for $\gamma_0, \dots, \gamma_{N+1} \in \R$.
\end{enumerate}
Similarly, the following maps fulfill the set of Assumptions~\ref{as:B}:
\begin{enumerate}
    \item $c(x)= c_0 + \sum_{k=1}^N c_k\lvert x_k\rvert^{\frac{2}{k}}$ for $c_0, \dots, c_N \in \R$;
    \item $b(x)= b_0 + \sum_{k=1}^N b_k\lvert x\rvert^{\frac{1}{k}}$ for $b_0, \dots, b_N \in \R$;
    \item $\gamma(x)= \sum_{k=1}^N \gamma_k\lvert x\rvert^{\frac{1}{k}}$ for $\gamma_1, \dots, \gamma_N \in \R$.
\end{enumerate}
\end{example}

\begin{proposition}\label{prop:existence_primal_ode}
Under the Assumptions~\ref{as:A} or Assumptions~\ref{as:B} on the maps $b, \beta, c, \gamma, \Gamma: \R^N \rightarrow \R$ and denoting $\bar{\Z}_t:=(1, Z_t, \dots, Z_t^N)^{\mathsf{T}}$, the process $(\E[\bar{\Z}_t])_{t\geq 0}$ satisfies the following $N+1$ dimensional non-linear ODE: 
\begin{equation}\label{eq:ode_L}
\bar{z}'(t)= \mathbf{L}(z(t))\bar{z}(t) \textrm{ and } \bar{z}(0) = \E[\bar{\Z}_0]
\end{equation}
with $\bar{z}(t)= (1, z(t))$ for $z(t)\in \R^N$ such that $\bar{z}(t)_0\equiv 1$, $\bar{z}(t)_i = z(t)_i$ for $i\in \{1, \dots, N\}$  and where for $j,k= 0, \dots, N$ 
\begin{equation*}
\mathbf{L}(z(t))_{kj}= \begin{cases} k\cdot \beta(z(t)) + \frac{k(k-1)}{2}\Gamma(z(t)) \,&\textrm{ if } k=j, \, k\neq 0 \\
k\cdot b(z(t))+ \frac{k(k-1)}{2}\gamma(z(t)) \,& \textrm{ if }k=j+1 \\
\frac{k(k-1)}{2}c(z(t)) \,& \textrm{ if }k=j+2\\
0 &\textrm{ else. }
\end{cases}
\end{equation*}
Moreover, for a fixed initial condition $z(0) \in\R^N$, the ODE given in \eqref{eq:ode_L} has a unique solution (up to its maximal time of existence) if $b, \beta, c, \gamma, \Gamma$ are locally Lipschitz on an open set $A\subseteq \R^N$ with $z(0)\in A$, in short we write $b, \beta, c, \gamma, \Gamma \in \textrm{Lip}(A,\R)$ for the local Lipschitz continuity requirement. 
\end{proposition}

Before we proceed to the proof, let us note for the statement regarding existence and uniqueness of solutions, we can start from any $z(0)\in \R^N$. However, for this solution to describe the process $(\E[\Z_t])_{t\geq0}$, we can only consider starting points of the form $z(0)= (z_0, z_0^2, \dots, z_0^N)$ for some $z_0\in \R$.

\begin{proof}[Proof of Proposition~\ref{prop:existence_primal_ode}] 

By Lemma~\ref{lem:new_ass_maps} the process $(M^{(k)}_t)_{t\geq0}$ is a martingale for all $k\in \{1, \dots, N\}$ and hence the form of \eqref{eq:ode_L} follows by an application of Itô's lemma. Moreover, the existence and uniqueness of solutions follows from Theorem~\ref{thm:sol_ode}.
\end{proof}

\begin{proposition}\label{prop:global_sol_ode}
Note that, if $b, \beta, c, \gamma, \Gamma$ are locally Lipschitz on $\R^N$ and either satisfy the Assumptions~\ref{as:A} or the Assumptions~\ref{as:B}, then the solution to \eqref{eq:ode_L} is necessarily global. 
\end{proposition}

We present the proof of Proposition~\ref{prop:global_sol_ode} in Appendix~\ref{app:global_sol_ode}.

Note that the maps given in Example~\ref{ex:maps_primal} are not locally Lipschitz on all of $\R^N$ which is required in Proposition \ref{prop:global_sol_ode} to get a global solution to ~\eqref{eq:ode_L}.
However, the following corollary nevertheless allows us to 
this under some different conditions that are satisfied by the maps in Example~\ref{ex:maps_primal}.

\begin{corollary}
Let $b, c, \gamma: \R^{N} \rightarrow \R_{\geq 0} $ satisfy the conditions of Assumptions~\ref{as:A}  or Assumptions~\ref{as:B}. If $z(0)>0$, additionally $b, c, \gamma \in \mathrm{Lip}(\R^N_{>0}; \R_{\geq 0})$ and $\beta, \Gamma$ are constant, then the solution to \eqref{eq:ode_L} is global. 
\end{corollary}

\begin{proof}
Let us first note that by assumption $z(0)>0$. Moreover, it holds that 
\begin{align*}
z_1(t) &= z_1(0)+ \int_0^t b(z(s)) ds+ \int_0^t \beta\cdot z_1(s) ds \\ &\geq z_1(0)+ \int_0^t \beta\cdot z_1(s) ds,
\end{align*}
where we used the fact that $b(\cdot)$ is non-negative. Hence, we can use  Gronwall's inequality for $$ -z_1(t) \leq -z_1(0) + \int_0^t \beta \cdot (-z_1(s)) ds$$ and thereby obtain $$ -z_1(t)\leq -z_1(0)\cdot \exp(\beta t) \Leftrightarrow z_1(t)\geq z_1(0)\cdot \exp(\beta t).
$$
Therefore, we have shown that if $z_1(0)>0$ the  first component stays positive for all $t\geq 0$. Similarly, this holds for higher components. We show this by induction and therefore assume it holds up to some $1\leq k\leq N-1 $ that $(z_i(t))_{1\leq i \leq k}>0$ for all $t\geq 0$, then it holds 
\begin{align*}
z_{k+1}(t) &= z_{k+1}(0) + \int_0^t k b(z(s))z_k(s) + \frac{k(k-1)}{2}\big(c(z(s))z_{k-1}(s)+ \gamma(z(s))z_k(s)\big) ds \\  &\quad +\int_0^t \left(k\beta + \frac{k(k-1)}{2}\Gamma\right)z_{k+1}(s) ds \\ &\geq z_{k+1}(0) + \int_0^t \left(k\beta + \frac{k(k-1)}{2}\Gamma\right)z_{k+1}(s) ds. 
\end{align*}
Again, using Gronwall's inequality, we obtain
$$z_{k+1}(t)\geq z_{k+1}(0)\exp\left(t\left(\frac{k(k-1)}{2}\Gamma + k\beta\right)\right).$$
\end{proof}

\begin{theorem}\label{thm:poly_MV}
Let us consider a McKean-Vlasov SDE of the form 
\begin{align}\label{eq:Poly_MV}
dZ_t &= \left(b(\E[\Z_t]) + \beta(\E[\Z_t]) Z_t \right)dt + \sqrt{c(\E[\Z_t])+ \gamma(\E[\Z_t])Z_t+\Gamma(\E[\Z_t]) Z_t^2}dW_t
\end{align}
for some $N\geq1$, fixed $\E[\Z_0]\in \R^N$, continuous maps $b,\, \beta, \,c, \,\gamma, \, \Gamma: \R^N \rightarrow \R$ and a Brownian motion $(W_t)_{t\geq 0}$. If $b, \beta, c, \gamma, \Gamma$ are locally Lipschitz continuous on $\R^N$, satisfy either the set of Assumptions~\ref{as:A} or Assumptions~\ref{as:B},
and the term under the square-root is non-negative, then there exists a unique strong solution $(Z_t)_{t\geq 0}$. Moreover, $(Z_t)_{t\geq 0}$ is a time-inhomogeneous polynomial process.
\end{theorem}

\begin{proof}
    Let $(z(t))_{t\geq 0}$ be the unique global solution to \eqref{eq:ode_L} with $z(0)= \E[\Z_0]$. Note that by the assumptions on $b, \beta, c, \gamma, \Gamma$ in the theorem such a global solution exists by Proposition \ref{prop:global_sol_ode}. Let us now study the following SDE: 
    \begin{align}\label{eq:poly_MV_SDE_helper}
        dZ_t = \big(b(z(t))+ \beta(z(t))Z_t\big)dt + \sqrt{c(z(t)) + \gamma(z(t))Z_t + \Gamma(z(t))Z^2_t} dW_t. 
    \end{align}
    We now show that a solution to the martingale problem associated with~\ref{eq:poly_MV_SDE_helper} exists, by the assumptions on $b, \beta, c, \gamma, \Gamma$. Let us elaborate on this point for the two cases of Assumptions~\ref{as:A} and Assumptions~\ref{as:B} separately. 

    Under the set of Assumptions~\ref{as:A} we have established in the proof of Proposition~\ref{prop:global_sol_ode} that for all $t\in [0,\infty)$ it holds that 
    $$ \sup_{s\in [0,t]} \lvert z_1(s) \rvert \leq h_1(t) \textrm{ and } \|z(t)\|_N^{2N} \leq h_{2N}(t).$$
    This allows us to verify the growth requirements stated in Theorem~\ref{thm:sol_MP} for \eqref{eq:poly_MV_SDE_helper}: 
    \begin{align*}
    \sigma^2(t, x) &= c(z(t))+ \gamma(z(t))x + \Gamma(z(t))x^2 \\ &\leq c_0\Big(1+ f_c\big(h_1(t)\big) + h_{2N}(t)\Big) + \gamma_0 \Big(1+ f_{\gamma}\big(h_1(t)\big) + h_{2N}(t)\Big)x + \Gamma_0 x^2 \\
    x B(t, x)& = b(z(t))x + \beta(z(t))x^2 \\& \leq b_0\left(1+ h_1(t)\right)x + \beta_0 x^2
    \end{align*}
    and indeed the growth conditions~\eqref{eq:xb_growth},~\eqref{eq:sigma_growth} are satisfied for $h(t)= f_c(h_1(t)) + h_{2N}(t)+ h_1(t)$, $K^{\sigma}= c_0+\gamma_0+ \Gamma_0$ and $K^{B}= b_0 + \beta_0$. 

    The arguments in the case of Assumptions~\ref{as:B} are very similar, recalling we have established $\|z(t)\|_N^{2N}\leq g_{2N}(t)$ in the proof of Proposition~\ref{prop:global_sol_ode} and hence for~\eqref{eq:poly_MV_SDE_helper} it holds that 
    \begin{align*}
    \sigma^2(t, x) &= c(z(t))+ \gamma(z(t))x + \Gamma(z(t))x^2 \\ &\leq c_0\big(1+ g_{2N}(t)\big) + \gamma_0 \big(1+ g_{2N}(t)\big)x + \Gamma_0 x^2 \\
    x B(t, x)& = b(z(t))x + \beta(z(t))x^2 \\& \leq b_0\left(1+ g_{2N}(t)\right)x + \beta_0 x^2.
    \end{align*}
    In this case, the growth conditions~\eqref{eq:xb_growth},~\eqref{eq:sigma_growth} are satisfied for $h(t)= g_{2N}(t)$, $K^{\sigma}= c_0+\gamma_0+ \Gamma_0$ and $K^{B}= b_0 + \beta_0$. The assumption 
    that $c, \gamma, \Gamma$ are such that the term under the square-root is non-negative ensures that the PMP holds. Therefore, there exists a solution to the martingale problem associated with~\eqref{eq:poly_MV_SDE_helper}. Moreover, $\sigma^2(t, x), B(t,x)$ fulfil the local Lipschitz continuity requirement of Theorem~\ref{thm:sol_MP}. Therefore, Theorem~\ref{thm:sol_MP} is applicable to \eqref{eq:poly_MV_SDE_helper} and we conclude that \eqref{eq:poly_MV_SDE_helper} admits a unique strong solution. 
    
    Applying Itô's lemma to the solution $(Z_t)_{t\geq0}$ of \eqref{eq:poly_MV_SDE_helper} and noting that the local martingales are true martingales due to the growth-bounds on $b, \beta, c, \gamma, \Gamma$, it is apparent that $\E[\Z_t]$ is a solution to \eqref{eq:ode_L}. On the other hand, $(z(t))_{t\geq0}$ is the unique global solution to \eqref{eq:ode_L} and therefore $\E[\Z_t] = z(t)$ for all $t\in [0, \infty)$. Hence, we can conclude that \eqref{eq:poly_MV_SDE} and \eqref{eq:poly_MV_SDE_helper} coincide, and we have found a unique strong solution to \eqref{eq:Poly_MV}. Moreover, the solution of \eqref{eq:Poly_MV} clearly satisfies the definition of a time-inhomogeneous polynomial process, see Definition~\ref{def:poly_diffusion}. 
\end{proof}

\begin{remark}\label{rem:poly_MV_S}
Concrete sufficient conditions on the maps $c, \gamma, \Gamma$ for the term under the square-root to be positive are: 
\begin{itemize}
\item $c\geq 0$, $\Gamma \geq 0$ and $\gamma\equiv 0$, or, 
\item the term under the square-root is of the form $c(\cdot) + \gamma(\cdot) z + \Gamma(\cdot)z^2 = (\tilde{\gamma}(\cdot) + \tilde{\Gamma}(\cdot)z)$.
\end{itemize}
\end{remark}
In addition to the general case of Theorem~\ref{thm:poly_MV} where the state space is $\mathbb{R}$, we can study the state-spaces $\R_{\geq 0}$ and $[0,1]$. 

\begin{corollary}\label{coro:poly_MV_S}
Let $S=\R_{\geq 0}$ or $S=[0,T]$ and $\E[\Z_0]\in S$. Under the set of Assumptions~\ref{as:A} or Assumptions~\ref{as:B} and if $b, \beta, c, \gamma, \Gamma$ are locally Lipschitz continuous on $\R_{\geq0}$, sufficient conditions for existence of a strong solution to \eqref{eq:Poly_MV} on the respective state spaces are:
\begin{itemize}
\item $S= \R_{\geq 0}$: in this case a sufficient condition is $b: \R_{\geq 0} \rightarrow \R_{\geq 0}$ and $c\equiv 0$, $\gamma, \Gamma: \R_{\geq 0} \rightarrow \R_{\geq 0}$.

\item $S = [0,1]$: here sufficient conditions are $b: \R_{\geq 0} \rightarrow \R_{\geq 0}$, $b + \beta: \R_{\geq 0} \rightarrow \R_{\leq 0}$, and $c\equiv 0$, $\gamma \equiv -\Gamma \geq 0$.  
\end{itemize}
In each case, the solution to \eqref{eq:Poly_MV} is a time-inhomogeneous polynomial process on the respective state-space $S$. 
\end{corollary}

\begin{proof}
Note that by Assumptions~\ref{as:A} and~\ref{as:B} respectively, the growth-conditions of the SDE's coefficients are unchanged to the case of Theorem~\ref{thm:poly_MV}. The only conditions that differ in this situation are the ones relating to the PMP. Indeed, additional conditions above correspond to the requirements given in Proposition~\ref{prop:pmp} for the PMP to hold. Then, existence of a unique strong solution follows directly by Theorem~\ref{thm:sol_MP}. 

\end{proof}

\begin{corollary}[Moment Formula] \label{coro:primal_moment_formula}
	Let $(Z_t)_{t\geq 0}$ be the solution of ~\eqref{eq:Poly_MV} and the assumptions of Theorem~\ref{thm:poly_MV} be satisfied. Recall the notation $\bar{\Z}_t= (1, Z_t, Z_t^2, \dots, Z_t^N)^{\mathsf{T}}$ for all $t\geq 0$. Note that for $u= (u_0, \dots, u_N)^{\mathsf{T}}\in\R^{N+1}$ we can denote polynomials in $Z_t$ by $\langle \bar{\Z}_t, u\rangle= u_0 + u_1 Z_t + \dots + u_N Z^N_t$. Then, the following moment formula holds on any arbitrary but fixed time horizon $[0,T]$: 
	\begin{equation*}
		\E[\langle \bar{\Z}_t, u\rangle \lvert \mathcal{F}_s] = (1, Z_s, Z^2_s, \dots, Z^N_s) \cdot e^{\mathbf{\Omega}(s,t)}\cdot u,
	\end{equation*}
	any $0\leq s\leq t \leq T$, any $u\in\R^{N+1}$ and where $\mathbf{\Omega}(s,t)= \sum_{k=0}^{\infty} \mathbf{\Omega}_k(s,t)$ is the Magnus expansion $\bm{\mathcal{E}}(t)$ with $\bm{\mathcal{E}}(t)= \varepsilon\cdot \mathbf{L}(\E[\bar{\Z}]_t)^{\mathsf{T}}$ for $\varepsilon >0$ such that $\int_0^T \| \bm{\mathcal{E}}(s)\|_2 ds < \pi$.
\end{corollary}

\begin{proof}
	See Theorem~\ref{thm:moment_formula}. To find the infinitesimal generators $(\mathbf{H}_s)_{s\geq 0}$ of the process and show that $\mathbf{H}_s =\mathbf{L}(\E[\bar{\Z}]_s)^{\mathsf{T}}$ 
for all $s\geq 0$ we apply Proposition~\ref{prop:poly_diffusion} to elements of a polynomial basis. In particular, let us consider the basis $(v_0, \dots, v_N)$ with $v_i:= (x\mapsto x^i)$ for all $i \in \{0, \dots, N\}$. Then, applying Proposition~\ref{prop:poly_diffusion} we find for $k\geq 2$ that 
    \begin{align*}
        (\mathbf{H}_sv_k)(x) =& v_{k-2}(x)\cdot\left(\frac{k(k-1)}{2}c(\E[\Z_t]) \right)+  v_{k-1}(x)\cdot\left( k\cdot b(\E[\Z_t]) + \frac{k(k-1)}{2}\gamma(\E[\Z_t])\right) \\ &+ v_k(x) \left( k\cdot \beta(\E[\Z_t]) + \frac{k(k-1)}{2} \Gamma(\E[\Z_t])\right).
    \end{align*}     
    Recalling the convention we introduced in~\eqref{eq:transform_operator}, it follows that $(\mathbf{H}_s)_{jk} = \mathbf{L}(\E[\Z_s])_{kj}$ for all $k,j \in \{0, \dots, N\}$. Having found the infinitesimal generators, the rest of the claim follows directly from and application of Theorem~\ref{thm:moment_formula}.
\end{proof}

\begin{remark}
Note that the above expression for $\E[\langle \bar{\Z}_t, u \rangle \lvert \mathcal{F}_s]$ can also be obtained by solving Equation~\ref{eq:ode_L} with the initial value given by $\bar{z}(0)=\big(1, Z_s(\omega), \ldots, Z_s(\omega)^N\big)^{\mathsf{T}}$.
\end{remark}

\subsection{Dual point of view}\label{subsec:dual_poly}
We now approach polynomial McKean-Vlasov SDEs using a dual approach. Recall our notation $\bar{\Z}_t= (1, Z_t, Z_t^2, \dots, Z^N_t)^{\mathsf{T}}$ for all $t\geq 0$. We consider a process $(Z_t)_{t\geq 0}$ given by  SDEs of the form 
\begin{align}\label{eq:dual_SDE}
    dZ_t&= \left(\tilde{b}\left(\langle \bar{\Z}_0, \cc(t,0,e_0)\rangle, \dots, \langle \bar{\Z}_0, \cc(t,0,e_N)\rangle\right)+ \tilde{\beta}(\langle \bar{\Z}_0, \cc(t,0,e_0)\rangle, \dots, \langle \bar{\Z}_0, \cc(t,0,e_N)\rangle)Z_t\right)dt \nonumber \\& \quad+ \Big(\tilde{c}(\langle \bar{\Z}_0, \cc(t,0,e_0)\rangle, \dots, \langle \bar{\Z}_0, \cc(t,0,e_N)\rangle)+ \tilde{\gamma}(\langle \bar{\Z}_0, \cc(t,0,e_0)\rangle, \dots, \langle \bar{\Z}_0, \cc(t,0,e_N)\rangle)Z_t \nonumber\\ & \qquad+ \tilde{\Gamma}(\langle \bar{\Z}_0, \cc(t,0,e_0)\rangle, \dots, \langle \bar{\Z}_0, \cc(t,0,e_N)\rangle)Z_t^2 \Big)^{\frac{1}{2}}dW_t 
\end{align}
where $Z_0=z_0\in \R$ and $\cc: \Delta\times \R^{N+1}\rightarrow \R^{N+1}$ is continuous and $(e_0, \dots, e_N)$ the canonical basis of $\R^{N+1}$. To simplify notation and improve readability, let us write for all $t\geq 0$
\begin{equation*}
\mathbf{vec}^{\Z_0,\cc}(t) := \begin{pmatrix} \langle \bar{\Z}_0, \cc(t,0,e_0)\rangle\\ \langle \bar{\Z}_0, \cc(t,0,e_1)\rangle\\ \vdots\\ \langle \bar{\Z}_0, \cc(t,0,e_N)\rangle
\end{pmatrix}.
\end{equation*}
In this notation,~\eqref{eq:dual_SDE} reads: 
\begin{align} \label{eq:dual_sde_alt}
dZ_t&= \left(\tilde{b}\big(\mathbf{vec}^{\Z_0,\cc}(t)\big)+ \tilde{\beta}\big(\mathbf{vec}^{\Z_0,\cc}(t)\big)Z_t\right)dt \nonumber  \\&\quad+\sqrt{\tilde{c}\big(\mathbf{vec}^{\Z_0,\cc}(t)\big)+ \tilde{\gamma}\big(\mathbf{vec}^{\Z_0,\cc}(t)\big)Z_t + \tilde{\Gamma}\big(\mathbf{vec}^{\Z_0,\cc}(t)\big)Z_t^2 }dW_t 
\end{align}

For such an SDE, we can now formulate the following theorem:\\

\begin{theorem}\label{thm:dual_moment_formula}
    
Assume that a unique solution to the martingale problem associated with \eqref{eq:dual_SDE} exists and that the maps $\cc(\cdot)$ are the unique solution to the following ODEs
\begin{align}\label{eq:fwd_ode_c}
 \frac{\d}{\d t}\begin{pmatrix} \cc(t,0,e_0), \dots,\cc(t,0,e_N)\end{pmatrix} =  \begin{pmatrix} \cc(t,0,e_0), \dots,\cc(t,0,e_N)\end{pmatrix} \cdot \tilde{\mathbf{H}}(\mathbf{vec}^{\Z_0, \cc}(t))
\end{align}
\begin{align}\label{eq:bwd_ode_c}
 \frac{\d}{\d t}\begin{pmatrix} \cc(T,t,e_0), \dots,\cc(T,t,e_N)\end{pmatrix} = -\tilde{\mathbf{H}}(\mathbf{vec}^{\Z_0, \cc}(t))  \cdot \begin{pmatrix} \cc(T,t,e_0), \dots,\cc(T,t,e_N)\end{pmatrix}
\end{align}
with initial condition $\cc(0, 0, e_i)=e_i$ terminal condition $\cc(T,T, e_i)=e_i$ and where the matrix $\tilde{\mathbf{H}}$ 
is given by 
\begin{equation*}
\tilde{\mathbf{H}}(\mathbf{vec}^{\Z_0,\cc}(t))_{kj}= \begin{cases} \vspace{0.3cm} j\cdot \tilde{\beta}\left(\mathbf{vec}^{\Z_0,\cc}(t)\right) + \frac{j(j-1)}{2}\tilde{\Gamma}\left(\mathbf{vec}^{\Z_0,\cc}(t)\right), \,&\textrm{ if } k=j, \, j\neq 0 \\ \vspace{0.3cm}
j\cdot \tilde{b}\left(\mathbf{vec}^{\Z_0,\cc}(t)\right)+ \frac{j(j-1)}{2}\tilde{\gamma}\left(\mathbf{vec}^{\Z_0,\cc}(t)\right), \,& \textrm{ if }j=k+1 \\ \vspace{0.3cm}
\frac{j(j-1)}{2}\tilde{c}\left(\mathbf{vec}^{\Z_0,\cc}(t)\right), \,& \textrm{ if }j=k+2\\
0 &\textrm{ else }
\end{cases}
\end{equation*}
for $j,k = 0, \dots, N$. 

Setting for $u= (u_0, \dots, u_N)^{\mathsf{T}} \in \R^{N+1}$, $\cc(t,s, u) := \sum_{i=0}^N u_i \cc(t,s,e_i)$ for all $0\leq s\leq t\leq T$,  then $$\mathcal{M}_t:=\langle \bar{\Z}_t, \cc(T,t,u)\rangle $$ is  a martingale on $[0,T]$ for all $T\geq 0$ for all $u\in \R^{N+1}$ and therefore 
\begin{equation*}
\E[\mathcal{M}_T\lvert \mathcal{F}_t]= \E[\langle \bar{\Z}_T, u\rangle \lvert \mathcal{F}_t]= \langle \bar{\Z}_t, \cc(T,t,u)\rangle.
\end{equation*}
\end{theorem}

The proof of Theorem~\ref{thm:dual_moment_formula} is given in Appendix~\ref{app:dual_moment_formula}.

\begin{remark}
    Considering the SDEs given by~\ref{eq:dual_SDE} under the assumptions of Theorem~\ref{thm:dual_moment_formula}, then these SDEs are in fact of McKean-Vlasov type, since the coefficients $\tilde{b}, \tilde{\beta}, \tilde{c}, \tilde{\gamma}, \tilde{\Gamma}$ depend on the moments of $(Z_t)_{t\geq 0}$ due to 
    $$
\langle \bar{\mathbb{Z}}_0,\cc(t,0,e_i)\rangle= \mathbb{E}[Z_t^i], \quad i \in \{0, \ldots, N\}.$$
    Relating this to the vectorized notation, this means that $\mathbf{vec}^{\Z_0,\cc}(t)_i = \E[Z^i_t]$ for all $t\in [0, \infty)$ and $0\leq i \leq N$.
    However, we have neither established if and when (global) solutions for the ODEs~\eqref{eq:fwd_ode_c} and \eqref{eq:bwd_ode_c} exist, nor if and when we have a unique solution to the associated martingale problem. 
    Both of which will impose conditions on the maps $\tilde{b}, \tilde{\beta}, \tilde{c}, \tilde{\gamma}, \tilde{\Gamma}$. Note that this the existence of solutions to~\eqref{eq:bwd_ode_c} and~\eqref{eq:fwd_ode_c} is far from trivial, because $\mathbf{vec}^{\Z_0, \cc}(t)$ depends on all of $c(t, 0, e_0), \dots, c(t, 0, e_N)$. 
    
    At this point, we would also like to emphasize again the connection between $\tilde{\mathbf{H}}$ and $\mathbf{L}$, namely that structurally  $\tilde{\mathbf{H}}= \mathbf{L}^{\mathsf{T}}$, which should be even more clear now. 
\end{remark}

\begin{proposition}\label{prop:existence_dual_ode}
   Assume all the maps $$\tilde{b}, \tilde{\beta}, \tilde{c}, \tilde{\gamma}, \tilde{\Gamma}: \R^{N+1} \rightarrow \R$$ to be locally Lipschitz-continuous on some set $A$ such that $\Z_0\in A$. Then a unique solution to \eqref{eq:fwd_ode_c}, \eqref{eq:bwd_ode_c} exists up to its maximal lifetime. Moreover, if $A=\R^N$ and the maps $\tilde{b}, \tilde{\beta}, \tilde{c}, \tilde{\gamma}, \tilde{\Gamma}$ either satisfy the set of Assumptions~\ref{as:A} or the set of Assumptions~\ref{as:B}, then the unique solution to \eqref{eq:fwd_ode_c} is global and hence the same holds for the solution to \eqref{eq:bwd_ode_c}. If, in addition to $A=\R^N$ and Assumptions~\ref{as:A} or~\ref{as:B}, the term under the square-root in~\eqref{eq:dual_SDE} remains non-negative, there exists a unique strong solution to~\eqref{eq:dual_SDE}.
\end{proposition}

\begin{proof}
Let us fix an arbitrary $Z_0 \in \R$ without loss of generality and note that, using~\eqref{eq:fwd_ode_c} we obtain the following ODE for $\mathbf{vec}^{\Z_0, \cc}$: 

\begin{equation}\label{eq:ode_vec}
  \frac{\d}{\d t} \mathbf{vec}^{\Z_0, \cc}(t)= \tilde{\mathbf{H}}(\mathbf{vec}^{\Z_0, \cc}(t))^{\mathsf{T}}\cdot \mathbf{vec}^{\Z_0, \cc}(t) \textrm{ with } \mathbf{vec}^{\Z_0, \cc}(0)=\bar{\Z}_0.  
\end{equation}

The existence of a unique solution to~\eqref{eq:ode_vec} up to its maximal lifetime is assured by the local Lipschitz conditions. Note, that~\eqref{eq:ode_vec} is the same system of ODEs as the one given in~\eqref{eq:ode_L} by the fact that $\mathbf{H}^{\mathsf{T}}= \mathbf{L}$. Hence, the existence of a global solution to~\eqref{eq:ode_vec} under the assumption that $A=\R^N$ and either Assumptions~\ref{as:A} or~\ref{as:B} being fulfilled, follows by an application of Proposition~\ref{prop:global_sol_ode}. Having obtained a solution to~\eqref{eq:ode_vec} given by $\big(\mathbf{vec}^{\Z_0, \cc}(t)\big)_{t\in [0, \infty)}$ we can plug this solution in to $\tilde{\mathbf{H}}(\cdot)$ to study~\eqref{eq:fwd_ode_c} and~\eqref{eq:bwd_ode_c} respectively to solve for $\big((\cc(T,t,e_i))_{i=0}^{N}\big)_{T\geq t\geq 0}$. Indeed, it holds that the right-hand-side of \eqref{eq:fwd_ode_c} is Lipschitz-continuous in $\big((\cc(t,0,e_i))_{i=0}^N\big)_{t\in [0, \infty)}$. To see this, recall the form of~\eqref{eq:fwd_ode_c}. Having obtained that $\langle \bar{\Z}_0, \cc(t,0,e_i)\rangle < \infty$ for all $i=0, \dots, N$ and $t\in \R_{\geq 0}$ for arbitrary $Z_0\in \R$, it follows that $\|\cc(t,0,e_i)\| < \infty $ all $i=0, \dots, N$ and $t\in \R_{\geq 0}$.
Similarly, for the backward ODE, we can simply plug in $(\tilde{\mathbf{H}}(\mathbf{vec}^{\Z_0, \cc}(t)))_{0\leq t \leq T}$ and recalling \eqref{eq:bwd_ode_c} it is clearly Lipschitz again, due to the bounds on $\tilde{\mathbf{H}}(\mathbf{vec}^{\Z_0, \cc}(t))$. Since $\tilde{\mathbf{H}}(\mathbf{vec}^{\Z_0, \cc}(t))$ is upper-triangular, one can recursively bound the components of $\cc(T, t, e_i)$ for each $i=0, \dots, N$. More precisely, for each $i=0, \dots, N$ apply Gronwall's inequality first on $\cc(T,t,e_i)_N$ (i.e. the $N$-th component of $\cc(T,t,e_i)$) using the bound on the coefficients of $\tilde{\mathbf{H}}(\cdot)$ and next repeat the same for $\cc(T,t,e_i)_{N-1}$ and so on. We have therefore shown that a global unique solution exists to the forward and backward ODE. 

Let us now turn to the existence of a solution to the martingale problem and path-wise uniqueness. Let us note by~\eqref{eq:ode_vec} and the fact that in this notation~\eqref{eq:dual_SDE} corresponds to~\eqref{eq:dual_sde_alt}, we are in fact exactly in the same situation as in the proof of Theorem~\ref{thm:poly_MV}. Hence, the arguments for existence and path-wise uniqueness of a solution are exactly those presented in the proof of Theorem~\ref{thm:poly_MV}.
\end{proof}

\begin{corollary}
As in the primal case, we can study the existence of solution on other state-spaces than $S=\R$. Sufficient conditions for the positive maximum principle to hold on the respective state-spaces are: 
\begin{itemize}
\item for state-space $S= \R_{\geq 0}$:  $\tilde{b}: \R_{\geq 0} \rightarrow \R_{\geq 0}$ and $\tilde{c}\equiv 0$, $\tilde{\gamma}, \tilde{\Gamma}: \R_{\geq 0} \rightarrow \R_{\geq 0}$.

\item for the case of $S = [0,1]$:  $\tilde{b}: \R_{\geq 0} \rightarrow \R_{\geq 0}$, $\tilde{b} + \tilde{\beta}: \R_{\geq 0} \rightarrow \R_{\leq 0}$, and $\tilde{c}\equiv 0$, $\tilde{\gamma} \equiv -\tilde{\Gamma} \geq 0$.  
\end{itemize}
\end{corollary}

\begin{proof}
The proof is completely analogous to the one of Corollary~\ref{coro:poly_MV_S}.

\end{proof}

\begin{remark}\label{rem:diff_primal_dual}
Let us make the following comment on the difference between the primal and dual point of view: although we have ultimately used the same assumptions to establish existence of solutions, the derivation was profoundly different. In the primal case, we have required the set of Assumptions~\ref{as:A} or~\ref{as:B} \emph{to obtain} the system of ODEs~\eqref{eq:ode_L}. It then turned out, that these growth-assumptions are also sufficient conditions for the existence of global solutions. In the dual case, however, we obtained the system of ODEs~\eqref{eq:ode_vec} via the Kolmogorov forward and backward equations. We had thus far not imposed any growth-conditions. Those only came into play as sufficient conditions for the existence of global solutions to~\eqref{eq:ode_vec}. Although this difference may seem subtle, the following SDE shows it is not: 
\begin{equation}
dZ_t = \sqrt{\E[\Z^2_t]Z_t} dW_t \textrm{ with } Z_0=z>0.
\end{equation}
Clearly, $\gamma(x)= x_2$ violates both Assumptions~\ref{as:A} and~\ref{as:B}. Moreover, our approach to establish~\eqref{eq:ode_L} in the primal approach via Gronwall's inequalities fails. However, in the dual case, neither Assumptions~\ref{as:A} nor~\ref{as:B} are required to find~\ref{eq:ode_vec}. And although, in this case, $\gamma$ does not satisfy the sufficient conditions for the existence of a global solution, it is important to note that these are not necessary conditions. Indeed, in this case $N=2$ and the system of ODEs boils down to: 
\begin{align*}
\frac{\d}{\d t} (\mathbf{vec}^{\Z_0, \cc}(t))_1 &= 0 \textrm{ with } (\mathbf{vec}^{\Z_0, \cc}(0))_1 = z \\
\frac{\d}{\d t} (\mathbf{vec}^{\Z_0, \cc}(t))_2 &=
(\mathbf{vec}^{\Z_0, \cc}(t))_2 (\mathbf{vec}^{\Z_0, \cc}(t))_1 \textrm{ with } (\mathbf{vec}^{\Z_0, \cc}(0))_2 = z^2
\end{align*}
which admits a global unique solution 
$$ \mathbf{vec}^{\Z_0, \cc}(t) = \begin{pmatrix} z \\ z^2 e^{zt}
\end{pmatrix}.
$$
Using this it is straightforward to solve for $\cc$. Moreover, for the state-space $S= \R_{\geq 0}$ the PMP holds, and we can apply Theorem~\ref{thm:sol_MP} for $h(t)= e^{zt}$ to establish existence of path-wise unique solutions. \emph{Ex post} we also find that in this example $(Z_t)_{t\in [0, \infty)}$ is indeed a true martingale, due to the fact that we have an explicit expression for $(\E[Z^2_t])_{t\in [0, \infty)}$ and can therefore check square-integrability. 

Having found an explicit expression for the moments by solving the above ODE, we can conclude \emph{a posteriori} that the process is a true martingale, due to square integrability. 
\end{remark}

\subsection{Adding common noise}\label{subsec:common_noise}
 
In this section, we study a class of SDEs which include an additional \emph{common noise} term $( + \dots dW_t^0)$ and whose coefficients depend on the conditional expectation of the process, that is, conditional on the common noise. Again, we require a form which resembles that of polynomial processes to study this family of SDEs. Let us make this more precise in the following definition: 

\begin{definition}[Conditional Polynomial McKean-Vlasov SDEs]
Let $(W_t)_{t\geq 0}$ and $(W^0_t)_{t\geq 0}$ denote two independent Brownian motions and denote by $\Wcal^0_t$ the $\sigma$-algebra generated by $W^0_t$ for each $t\geq 0$. We slightly abuse notation and write for all $t\geq 0$
$$\E[\Z_t \lvert \Wcal^0_t] := (\E[Z_t \lvert \Wcal^0_t], \dots, \E[Z^N_t \lvert \Wcal^0_t])$$ for some arbitrary but fixed $N>0$ and some process $(Z_t)_{t\geq0}$. We call the following SDE of \emph{conditional polynomial McKean-Vlasov type}:
\begin{align}\label{eq:cond_poly_MV}
\nonumber dZ_t = & \Big(b(\E[\Z_t \lvert \Wcal^0_t]) + \beta(\E[\Z_t \lvert \Wcal^0_t]) Z_t \Big)dt \\ & + \sqrt{c(\E[\Z_t \lvert \Wcal^0_t]) + \gamma(\E[\Z_t \lvert \Wcal^0_t]) Z_t + \Gamma(\E[\Z_t \lvert \Wcal^0_t]) Z_t^2} dW_t\\ \nonumber
& + \Big(l(\E[\Z_t \lvert \Wcal^0_t]) + \Lambda(\E[\Z_t \lvert \Wcal^0_t]) Z_t \Big)dW^0_t
\end{align}
where $b, \beta, c, \gamma, \Gamma, l, \Lambda: \R^N \rightarrow \R$ are some continuous maps. 
\end{definition}

We would like to point out that a particular case of such a conditional polynomial McKean-Vlasov SDE has been studied in~\cite{Huber_2024} in the context of market capitalizations curve modeling. 

The idea to approach this class of processes is similar to the primal approach to polynomial McKean-Vlasov SDEs of Section \ref{subsec:primal}, where we used It\^o's lemma to find an ODE for $\E[\Z_t]$ and sufficient conditions for the existence of solutions to this system ODEs. Here, we aim to describe $\E[\Z_t\lvert \Wcal_t^0]$ by a system of SDEs via It\^o's lemma and determine sufficient conditions for the existence of solutions. 

Before we derive this system of SDEs, let us state the following lemma, 
 for which we present a proof in the appendix to to be self-contained (but of course it is not a new result).

\begin{lemma}\label{lem:fubini_type}
Let $(W_t)_{t\geq 0}$ and $(W^0_t)_{t\geq 0}$ be two independent Brownian motions, and write $(\Wcal^0_t)_{t\geq0}$ for the filtration generated by $(W^0_t)_{t\geq 0}$. Moreover, denote by $\mathbb{F}:= (\Fcal_t)_{t\geq 0}$ the filtration generated by both Brownian motions jointly. Then, for any locally square-integrable, $\mathbb{F}$-adapted  process $(H_t)_{t\geq 0}$ it holds that  
$$\E\left[ \int_0^t H_s dW_s  \Big\lvert \Wcal_t^0\right]= 0 \quad \textrm{ and } \quad  \E\left[ \int_0^t H_s dW^0_s  \Big\lvert \Wcal_t^0\right]= \int_0^t \E\left[ H_s \big\lvert \Wcal_t^0\right]dW^0_s   .$$
\end{lemma}

We are now ready to derive the system of SDEs which is fulfilled by $\E[\Z_t \lvert \Wcal_t^0]$. Note that, we do not state anything about existence of solutions yet, but postpone this to Proposition~\ref{prop:joint_poly_1},~\ref{prop:joint_poly_2} and Theorem~\ref{thm:cond_existence_MP} and the discussion thereafter.

Similar to the case without common noise, we can formulate assumptions under which we can establish existence of and uniqueness of solutions to~\eqref{eq:cond_poly_MV}. Note that in the case with common-noise, however, we only have one set of assumptions. This is because the procedure of first bounding the first component, as we did for the set of Assumptions~\ref{as:A} does not work in the case with common noise. 

\begin{assumpB}\label{as:C}
Recall the notation $\| \cdot \|_N$ as introduced in Definition~\ref{def:norm_N}. We require the maps $b, \beta, c, \gamma, \Gamma, l, \Lambda$ to satisfy the following for all $x \in \R^N$: 
\begin{itemize}
\item $\lvert \beta(x)\rvert\leq \beta_0$,  $\lvert \Gamma(x)\rvert \leq \Gamma_0$, $\lvert \Lambda(x) \rvert \leq \Lambda_0$ for some $\beta_0, \Gamma_0, \Lambda_0 \in \R$;

\item $\lvert b(x) \rvert \leq b_0 (1 + \|x\|_N)$, $\lvert \gamma(x) \rvert \leq \gamma_0 (1 + \|x\|_N)$ and $\lvert l(x) \rvert \leq l_0 (1 + \|x\|_N)$ for some $b_0, \gamma_0, l_0\in \R$; 

\item $\lvert c(x) \rvert \leq c_0 (1+ \|x\|^2_N)$  for some $c_0\in \R$.
\end{itemize}
\end{assumpB}

Using this set of assumptions, we are ready to state the following lemma, the proof of which can be found in Appendix~\ref{app:cond_poly_squareint}:

\begin{lemma}\label{lem:cond_poly_squareint}
Let us fix an arbitrary $N>0$ and assume there exists a solution $(Z_t)_{t\in[0,\infty)}$ to~\eqref{eq:cond_poly_MV}. Moreover, let the maps $b, \beta, c, \gamma, \Gamma, l,  \Lambda: \R^N \rightarrow \R$ satisfy Assumptions~\ref{as:C}. Then, for all $1\leq k\leq N$ it holds that 
\begin{align*}
d\E[Z_t^k \lvert &\Wcal_t^0]=\Bigg\{ \Big(k\cdot b(\E[\Z_t \lvert \Wcal_t^0]) + \frac{k(k-1)}{2}\big(\gamma(\E[\Z_t \lvert \Wcal_t^0]) \\ & \qquad + 2l(\E[\Z_t \lvert \Wcal_t^0])\Lambda(\E[\Z_t \lvert \Wcal_t^0])  \big) \Big) \E[Z^{k-1}_t \lvert \Wcal_t^0]\\ &\qquad  + \mathbf{1}_{\{k\geq2\}}\frac{k(k-1)}{2} \big(c(\E[\Z_t \lvert \Wcal_t^0]) + l^2(\E[\Z_t \lvert \Wcal_t^0])\big) \E[Z^{k-2}_t \lvert \Wcal_t^0] \\ & \qquad + \big( k\cdot \beta(\E[\Z_t \lvert \Wcal_t^0]) + \frac{k(k-1)}{2}(\Gamma(\E[\Z_t \lvert \Wcal_t^0]) + \Lambda^2(\E[\Z_t \lvert \Wcal_t^0])\big) \E[Z^k_t \lvert \Wcal_t^0] \Bigg\} dt \\ & \qquad + \left(l(\E[\Z_t \lvert \Wcal_t^0])\E[Z^{k-1}_t \lvert \Wcal_t^0] + \Lambda(\E[\Z_t \lvert \Wcal_t^0]) \E[Z^{k}_t  \lvert \Wcal_t^0] \right) dW_t^0.
\end{align*}
\end{lemma}

We are now finally ready to tackle the question of existence of solutions. To do so, we will establish further conditions on the structure of the drift and volatility conditions. Our first approach is to study, under which assumptions, the process $(Z_t, \E[ \Z_t \lvert \Wcal_t^0])_{t\geq0}$ is a multidimensional polynomial process. Note that we are now in a time-homogeneous setting, see~\cite{Cuchiero_2012} for polynomial processes that are time-homogeneous. We first formulate this for a general state-space $S$ under the assumption of the PMP being satisfied. Later, we will provide examples of state-spaces as well as sufficient conditions for the PMP to hold.

\begin{proposition}\label{prop:joint_poly_1}
Let us assume that the maps $b, \beta, \gamma, \Gamma, l, \Lambda$ are all constant, and denoted by ${b}_0, {\beta}_0,\gamma_0, \Gamma_0, l_0, \Lambda_0$ respectively, as well as $c(x) = {c}_0 + {c}_1 x_1 + {c}_2 x_2$ for some ${c}_0, {c}_1, {c}_2\in \R$ and all $x\in \R^N$.
If the PMP is satisfied on the desired state-space $S\subseteq \R^3$, then for $N=2$, $(Z_t, \E[\Z_t \lvert \Wcal_t^0])_{t\geq0}$ \eqref{eq:cond_poly_MV} is jointly polynomial on $S$ and there exists a path-wise unique solution. More explicitly, the process $(Z_t, \E[\Z_t \lvert \Wcal_t^0])_{t\geq0}$ satisfies the following system of SDEs:
\begin{align*}
&dZ_t =  \left({b}_0 + {\beta}_0 Z_t \right)dt  + \sqrt{{c}_0 + {c}_1 \E[Z_t \lvert \Wcal_t^0]+ {c}_2 \E[Z^2_t\lvert \Wcal_t^0] + {\gamma}_0 Z_t + {\Gamma}_0 Z_t^2} dW_t \\& \qquad + \left({l}_0 + {\Lambda}_0 Z_t \right)dW^0_t\\
&d\E[Z_t \lvert \Wcal_t^0]= \Big\{ {b}_0+ {\beta}_0 \E[Z_t\lvert \Wcal_t^0] \Big\}dt + \Big\{ {l}_0 + {\Lambda}_0 \E[Z_t \lvert \Wcal_t^0]\Big\} dW_t^0
\\
&d\E[Z^2_t \lvert \Wcal_t^0]= \Big\{{c}_0 +{l}^2_0+  
(2{b}_0 + {\gamma}_0 + 2{l}_0{\Lambda}_0+ {c}_1) \E[Z_t \lvert \Wcal_t^0] \\ &\qquad \qquad \qquad+ (2{\beta}_0 + {\Gamma}_0 + {\Lambda}_0^2+ {c}_2) \E[Z^2_t \lvert \Wcal_t^0] \Big\} dt +\Big\{{l}_0 \E[Z_t\lvert \Wcal_t^0]+ {\Lambda}_t^0 \E[Z^2_t\lvert \Wcal_t^0] \Big\} dW_t^0
\end{align*}
\end{proposition}

\begin{proof}
Under these assumptions, it is straightforward to check that the system of SDEs outlined in Lemma~\ref{lem:cond_poly_squareint} reduces to the system of SDEs stated in the claim. Moreover, 
the process $(Z_t, \E[Z_t\lvert \Wcal_t^0] , \E[Z^2_t \lvert \Wcal_t^0])_{t\geq 0}$ is clearly polynomial, which can for example be seen from the differential characteristics. Regarding the path-wise uniqueness, we first consider $(\E[Z_t\lvert \Wcal_t^0], \E[Z^2_t\lvert \Wcal_t^0])_{t\geq 0}$ alone. Indeed, the corresponding drift and diffusion term are locally Lipschitz continuous such that~\cite[Theorem 3.7]{EK} is applicable. Hence, a path-wise unique solution exists for $(\E[Z_t\lvert \Wcal_t^0], \E[Z^2_t\lvert \Wcal_t^0])_{t\geq 0}$. We can now plug this solution in and  solve for $(Z_t)_{t\geq 0}$ which is now a one-dimensional SDE and hence we can apply~\cite[Theorem 3.8]{EK} to establish path-wise uniqueness. 
\end{proof}

Alternatively, by changing the form of the map $c(\cdot)$ to not include a term $\E[Z^2_t\lvert \Wcal_t^0] $ but instead a term $(\E[Z_t\lvert \Wcal_t^0])^2$ we gain more flexibility for the form of the rest of the coefficients, since we now only need  $(Z_t, \E[Z_t\lvert \Wcal_t^0])_{t\geq 0}$ to be polynomial instead of $(Z_t, \E[Z_t\lvert \Wcal_t^0] , \E[Z^2_t \lvert \Wcal_t^0])_{t\geq 0}$. Let us make this more precise in the next proposition.

\begin{proposition}\label{prop:joint_poly_2}
Assume that the following hold for all $x\in \R^N$
\begin{itemize}
\item $\beta(x)\equiv \beta_0$, $\Gamma(x)\equiv \Gamma_0$, $l(x) \equiv {l}_0$, $\Lambda(x) \equiv {\Lambda}_0$ for some $\beta_0,\Gamma_0, {l}_0, \Lambda_0 \in \R$
\item $b(x) = {b}_0 + \tilde{b}_1 x_1$ and $\gamma(x) = {\gamma}_0+ {\gamma}_1 x_1$ for some ${b}_0, {b}_1, \gamma_0, \gamma_1 \in \R$;
\item $c(x)= {c}_0 + {c}_1x_1 + \tilde{c}_2\left( x_1\right)^2$ for some ${c}_0,{c}_1,{c}_2\in \R$.

\end{itemize}
If the PMP is satisfied on $S\subseteq \R^2$, then for $(Z_t)_{t\geq 0} $ given by \eqref{eq:cond_poly_MV}, $(Z_t, \E[Z_t \lvert \Wcal_t^0])_{t\geq0}$ is a polynomial process on $S$ given by the solution to the SDEs:
\begin{align*}
dZ_t &= \left({b}_0+ {b}_1\E[Z_t\lvert \Wcal_t^0]+ {\beta}_0Z_t \right) dt+ \left({l}_0 + {\Lambda}_0 Z_t\right)dW_t^0\\
&\quad +  \sqrt{{c}_0 + {c}_1 \E[Z_t \lvert \Wcal_t^0]+ {c}_2 \left(\E[Z_t\lvert \Wcal_t^0]\right)^2 + \left({\gamma}_0+ {\gamma}_1 \E[Z_t \lvert \Wcal_t^0]\right) Z_t + {\Gamma}_0 Z_t^2} dW_t 
\\
d\E[Z_t \lvert \Wcal_t^0]&= \left({b}_0+ ({b}_1+ {\beta}_0) \E[Z_t\lvert \Wcal_t^0]\right) dt + \left({l}_0 + {\Lambda}_0 \E[Z_t \lvert \Wcal_t^0] \right)dW_t^0.
\end{align*}
In particular, this solution is path-wise unique. 
\end{proposition}

\begin{proof}
It is straightforward to derive this system of SDEs from Lemma~\ref{lem:cond_poly_squareint} and to see that the solution is a time-homogenous polynomial process. Regarding path-wise uniqueness, we can proceed exactly as in Proposition~\ref{prop:joint_poly_1}. Solving first for $(\E[Z_t\lvert \Wcal_t^0])_{t\geq 0}$ and then for $(Z_t)_{t\geq0}$.
\end{proof}

\begin{remark}
The above results hold under the assumption that the PMP is satisfied, and we will provide explicit conditions to ensure that in Proposition~\ref{prop:S_joint_poly}.
\end{remark}

The fact that the above processes are jointly polynomial gives us access to the full polynomial theory, in particular the moment formula, which may be  desirable in applications.

Instead of requiring the processes to be jointly polynomial, 
we can of course also treat the general case of equation  
\eqref{eq:cond_poly_MV}, by showing existence of solutions via the associated martingale problem.  This is the purpose of the next theorem. Again, we first formulate existence on some state space $S$ under the assumption that the PMP holds and provide concrete state spaces as well as sufficient conditions for the PMP later. 

\begin{theorem}\label{thm:cond_existence_MP}
Let $N\geq 1$ and assume that Assumptions~\ref{as:C} are satisfied. Moreover, let $S\subseteq \R^{N+1}$ be a state space for which the PMP is fulfilled.  Then \eqref{eq:cond_poly_MV} admits a solution with values in $S$ in the sense of a solution to the martingale problem. 
\end{theorem}

We show the proof of Theorem~\ref{thm:cond_existence_MP} in Appendix~\ref{app:cond_existence_MP}. We are left to formulate sufficient conditions for the PMP to be satisfied on the desired state space. However, in contrast to Proposition~\ref{prop:pmp}, we are now in a multidimensional but time-homogeneous setting. The reasoning is the same as the one applied in Proposition~\ref{prop:pmp}. We nevertheless briefly summarize sufficient conditions for the multidimensional setting in the next lemma for completeness. 

\begin{lemma}\label{lem:pmp_multidim}
Let us consider a state space $S\subseteq \R^d$ such that $S= S^{(1)}\otimes S^{(2)} \otimes \dots \otimes S^{(d)}$ with $S^{(1)}, \dots, S^{(d)} \in \{\R, \R_{\geq 0}\}$ for a multidimensional SDE driven by $m\geq 1$ independent Brownian motions $(W_t)_{t\geq 0} = (W_t^1, \dots , W_t^m)_{t\geq 0}$ 
\begin{equation}
    dX_t = B(X_t) dt + \Sigma(X_t)dW_t
\end{equation}
where $B: \R^d \rightarrow  \R^d$ and $\Sigma: \R^d \rightarrow \R^{d\times m}$. This SDE has a generator given by
\begin{align*}
( \mathcal{H}f)(x) = \sum_{i=1}^d B(x)_i \partial_i f(x) + \frac{1}{2}\sum_{i,j=1}^d (\Sigma(x)\Sigma(x)^{\mathsf{T}})_{ij} \partial_{ij} f(x)
\end{align*}
for $f\in C^{\infty}_c(\R^d)$. Let $x_0\in S$ be such that $f(x_0) = \sup_{x\in S} f(x)$. Then $ \mathcal{H}$ satisfies the positive maximum principle on $S=\R^d$ directly. For all other cases, we need some additional conditions on the boundary. Let us denote the set $I(x_0):= \big\{i\in \{1, \dots, d\} \lvert (x_0)_i= 0 \textrm{ and } S^{(i)}= \R_{\geq 0}\big\}$. If $I(x_0)\neq \emptyset$, i.e. at least for one coordinate the maximum lies on the boundary, we require for all $i \in I(x_0)$ and $j\in \{1, \dots, d\}$
$$(\Sigma(x_0)\Sigma(x_0)^{\mathsf{T}})_{ij} = 0 \textrm{
and } B(x_0)_i \geq 0 .$$ 
\end{lemma}

\begin{proof}
The reasoning is completely analogous to the proof of Proposition~\ref{prop:pmp}.
Let us first treat the case that $x_0$ is in the interior of $S$. In this case we know by the definition of $f(x_0)$ being the maximum that 
$$\partial_i f(x_0) = 0 \textrm{ for all } i \in \{1, \dots, d\}$$
$$ \big(\partial_{ij} f(x_0)\big)_{1\leq i, j \leq d} \textrm{ is negative semidefinite.}$$
Hence, we can easily verify that $( \mathcal{H}f)(x_0) \leq 0$ in this case, noting that 
\begin{align*}
( \mathcal{H}f )(x_0) &= \frac{1}{2}\sum_{i,j=1}^d (\Sigma(x_0)\Sigma(x_0)^{\mathsf{T}})_{ij} \partial_{ij} f(x_0) = \frac{1}{2}\sum_{i,j, k=1}^d \Sigma(x_0)_{ik}\Sigma(x_0)^{\mathsf{T}}_{kj} \partial_{ij} f(x_0) \\ &= \frac{1}{2}\sum_{i, j,k=1}^d \Sigma(x_0)^{\mathsf{T}}_{kj} \big( \partial_{ji} f(x_0) \big) \Sigma(x_0)_{ik} \leq 0
\end{align*}
by the negative semidefinitness of the Hessian. Now for the cases that in some coordinates (potentially all) the maximum lies on the boundary, let us treat separately the sets of indices $I(x_0)$ and $\{1, \dots, d\} \setminus I(x_0)$. For all $i \in I(x_0)$ it holds that $\partial_i f(x_0) \leq 0$ and by the requirements of the theorem we find 
\begin{align*}
&\sum_{i\in I(x_0)} B(x_0)_i \partial_if(x_0) \leq 0 \\
&\sum_{\substack{i\in I(x_0)\\ j\in \{1, \dots, d\}}} (\Sigma(x_0)\Sigma(x_0)^{\mathsf{T}})_{ij} \partial_{ij}f(x_0)+ \sum_{\substack{j\in I(x_0)\\ i\in \{1, \dots, d\}}} (\Sigma(x_0)\Sigma(x_0)^{\mathsf{T}})_{ij} \partial_{ij}f(x_0) =0,
\end{align*}
where the second sum follows by symmetry. 

On the other hand, $\partial_i f(x_0)=0$ for all $i\notin I(x_0)$ and $(\partial_{ij} f(x_0))_{i,j \notin I(x_0)}$ needs to be negative semidefinite, as $f(x_0)$ is concave in these coordinates. Hence, it follows as above that 
\begin{align*}
\sum_{i\notin I(x_0)} B(x_0)_i \partial_i f(x_0) + \frac{1}{2}\sum_{{i, j\notin I(x_0)}} (\Sigma(x_0)\Sigma(x_0)^{\mathsf{T}})_{ij} \partial_{ij}f(x_0) \leq 0
\end{align*}
and we have proven the claim. 
\end{proof}

We now outline possible state-spaces and sufficient conditions for the PMP to be satisfied in the case of Proposition~\ref{prop:joint_poly_1},~\ref{prop:joint_poly_2} and Theorem~\ref{thm:cond_existence_MP}, respectively. Note that what we mean by "state-spaces" here is slightly different than what is sometimes referred to as the state-space of a Markov process in the literature. We mean that the process takes values in the respective space and will not leave the state-space in finite time. However, the process will not explore the entire space, and it is also not possible for the process to start at every point in the space. Namely, if $Z_0=z \in \R$ the process $(Z_t, \E[Z_t\lvert \Wcal_t^0], \E[Z_t^2 \lvert \Wcal_t^0])_{t\geq 0}$ can only start in points like $(z, z, z^2)$.  Keeping this in mind, we are now ready to state the following theorems:

\begin{proposition}[State-spaces and PMP for Proposition~\ref{prop:joint_poly_1}~\&~\ref{prop:joint_poly_2}]\label{prop:S_joint_poly}
We consider the state-space $S= \R_{\geq 0} \times \R_{\geq 0} \times \R_{\geq 0}$ and $S= \R_{\geq 0} \times \R_{\geq 0}$, respectively. In both cases, the PMP holds if $b, \gamma, {\Gamma}_0 \geq 0$
and ${l}_0= c= 0$. 
For the case $S= \R \times \R$ (Proposition~\ref{prop:joint_poly_2}) we require ${c}_0, {c}_2, {\Gamma}_0 \geq 0$ and $\gamma, {c}_1=0$. And in addition to that, it must hold ${l}_0={b}_0=0$ for the PMP to be satisfied on $\R\times \R \times \R_{\geq 0}$ in the context of Proposition~\ref{prop:joint_poly_1}.
\end{proposition}

\begin{proof}
 Let us recall the statements of Lemma~\ref{lem:pmp_multidim}. We can then treat the individual cases separately:

\begin{itemize}
    \item We start with Proposition~\ref{prop:joint_poly_1} and $S= \R_{\geq 0}\times \R_{\geq 0} \times \R_{\geq 0}$. In this case we require for $x_0 \in \R_{>0}\times \R_{>0} \times \R_{>0}$ that $\Sigma(x_0)$ is real-valued, i.e. the term under the square-root to be non-negative, which is indeed the case for $\gamma_0, \Gamma_0\geq0$ and $l_0=c=0$. Next we check the case that the boundary is attained in at least one coordinate, i.e. $I(x_0)$. It is easy to see that the relevant drift terms indeed point inwards and that the diffusion coefficients vanish at these points, if $b, \gamma, {\Gamma}_0 \geq 0$ and ${l}_0= c= 0$.

    \item In the case of Proposition~\ref{prop:joint_poly_1} and $S= \R\times \R \times \R_{\geq 0}$. We first check the interior, i.e. $x_0 \in \R\times \R \times \R_{\geq 0}$ where again the only requirement is for the term under the square-root to be non-negative, which indeed holds if $c_1=\gamma_0 =0$ and $\Gamma_0, c_0, c_2\geq 0$. In this case, the boundary can only be attained in the last coordinate, i.e. $I(x_0)=\{2\}$. Indeed, $B(x_0)_2\geq 0$ if $b_0=\gamma_0=l_0=c_1=0$ and $c_0\geq 0$ and $\Sigma(x_0)_{22} = 0$ if $l_0=0$. 

    \item Turning to Proposition~\ref{prop:joint_poly_2} and the case $S= \R \times \R$ the term under the square-root is indeed positive if $c_1=\gamma=0$ and $c_0, c_2, \Gamma_0\geq 0$. In this case there is no boundary to be checked. 

    \item Lastly, in the case of Proposition~\ref{prop:joint_poly_2} and $S= \R_{\geq 0} \times \R_{\geq 0}$ the term under the square-root is non-negative on all of $\R_{>0} \times \R_{>0}$ if $c, \gamma, \Gamma_0\geq 0$. If $0\in I(x_0)$ it holds that $B(x_0)_0\geq 0$ if $b\geq 0$ and the corresponding diffusion terms vanish if $l_0=c=0$. Alternatively, if $1 \in I(x_0)$, the drift $B(x_0)_1\geq 0$ if $b_0\geq 0$ and the diffusion vanishes if $l_0=0$.

\end{itemize}
\end{proof}

\begin{remark}
Let us point out the connection between the above statements and~\cite[Proposition 6.4]{Filipovic_Larsson}. In contrast to our considerations,~\cite[Proposition 6.4]{Filipovic_Larsson} prohibits absorption of the process by the boundary, while we allow the process to remain at the boundary. This difference leads to the requirement in~\cite[Proposition 6.4]{Filipovic_Larsson} that the drift remains strictly positive on the boundary. We are able to accommodate this for the case $S=\R_{\geq0} \times \R_{\geq 0}$ by enforcing $b, \beta >0$. If we allow for absorption at the boundary and therefore slightly weaken the conditions of~\cite[Proposition 6.4]{Filipovic_Larsson}, the conditions we found in Proposition~\ref{prop:S_joint_poly} are indeed also sufficient conditions for~\cite[Proposition 6.4]{Filipovic_Larsson}.
\end{remark}

\begin{proposition}[State-spaces and PMP for Theorem~\ref{thm:cond_existence_MP}]\label{prop:cond_existence_PMP}
First, we consider the state-space 
$S= \R_{\geq 0} \times \dots \times \R_{\geq 0}$ on which the PMP holds if $ l \equiv c \equiv 0$, $b, \gamma, \Gamma, \Lambda \geq 0$.
On the other hand, for the state-space $S= \R\times \R \times \R_{\geq0} \times \dots \times \R \times \R_{\geq 0}$ alternating or (or $S= \R\times \R \times \R_{\geq0} \times \dots \times \R \times \R_{\geq 0}\times \R$ alternating)  sufficient conditions for the PMP to hold are $\gamma \equiv 0$, $\Gamma, c\geq 0$ and for all $x_0 \in S $ it must hold that $c(x_0)=l(x_0)=0$, $b(x_0)\geq 0$.
Moreover, if $b, \beta, c, \gamma, \Gamma, l, \Lambda$ are locally Lipschitz-continuous, then the solutions are path-wise unique. 
\end{proposition}

\begin{proof}
We first study the case $S= \R_{\geq0} \times \dots \times \R_{\geq 0}$ and start with the interior, i.e. $x_0\in \R_{>0} \times \dots \times \R_{>0}$. Recalling~\eqref{eq:cond_poly_MV}, the term under the square-root is positive if $c, \gamma, \Gamma\geq 0$. Now turning to the boundary, let us consider $0\in I(x_0)$. Clearly, $B(x_0)_0 >0$ if $b\geq 0$ and the diffusion terms vanish if $l\equiv c\equiv 0$. For the cases $k\in I(x_0)$ for any $k\geq 1$, let us recall Lemma~\ref{lem:cond_poly_squareint}. Then $B(x_0)_k\geq 0$ if $b, \gamma, l, \Lambda\geq 0$ and $\Sigma(x_0)_{k2}=0$ if $l\equiv 0$.

Now for alternating state-spaces $S= \R\times \R \times \R_{\geq0} \times \dots \times \R \times \R_{\geq 0}$ the interior is $S^{\circ}= \R\times \R \times \R_{>0} \times \dots \times \R \times \R_{>0}$ which leads us to require $\gamma\equiv 0$ and $c, \Gamma \geq 0$ for the term under the square-root to be non-negative for all $x_0\in S$. Moreover, if some $k\in I(x_0)$ (where $k$ is even), we find $B(x_0)_k \geq 0$ if $\gamma(x_0)= c(x_0) =l(x_0)= 0$ and $b(x_0)\geq 0$ for all $x_0 \in \partial S$, whereas the diffusion term $\Sigma(x_0)_{k2}=0$ if $l(x_0) =0$ for all $x_0 \in \partial S$. 

We now turn to the claim of path-wise uniqueness. Let us recall the form of the SDE describing $(\E[\Z_t\lvert \Wcal_t^0])_{t\geq 0}$ as stated in Lemma~\ref{lem:cond_poly_squareint}. It is clear that if the maps $b, \beta, c, \gamma, \Gamma, l, \Lambda$ are locally Lipschitz-continuous, so are the associated drift and volatility terms and hence~\cite[Theorem 3.7]{EK} is applicable. 

Hence, the obtained path-wise unique solution of $(\E[\Z_t\lvert \Wcal_t^0])_{t\geq 0}$ can be plugged into the original SDE for $(Z_t)_{t\geq 0}$. While the volatility coefficient is not locally Lipschitz-continuous due to the square-root, this is now a one-dimensional SDE, and we can apply~\cite[Theorem 3.8]{EK} to conclude path-wise uniqueness. 
\end{proof}

\newpage


\newpage

\appendix

\section{ODE theory}

\begin{theorem}[Existence and Uniqueness of Solutions to non-linear ODEs, {see~\cite[page 171]{Hirsch_Morris}}]\label{thm:sol_ode}
Let us consider the following autonomous non-linear ODE:
\begin{equation*}
du_t = F(u_t)dt \quad u_0\in A\subseteq \R^n
\end{equation*}
where $u: [0,t^*) \rightarrow \R^n$ and $F: \R^n \rightarrow \R^n$. If $F$ is locally Lipschitz-continuous on $A$, i.e. $F\in \textrm{Lip}(A;\R^n)$, then there exists a unique solution $u: [0, t^*) \rightarrow \R^n$. We call the time $t^*$ the \emph{maximal time of existence} which is defined as $t^*:= \inf_{t\geq 0} \{u_t \in \partial A\}$ and $t^*=\infty$ if $u_t \in \partial A$ does not occur. In the latter case, we call a solution \emph{global}.
\end{theorem}

\section{Proofs}

\subsection{Proof of Theorem~\ref{thm:sol_MP}}\label{app:sol_MP}
\begin{proof}[Proof of Theorem~\ref{thm:sol_MP}]
The proof regarding the existence of solutions is analogous to the proof of \cite[Theorem 5.3.10 and Proposition 5.3.5]{EK}. The only difference is that we 
define $$f_n(t,x):= \varphi\left(\frac{t}{n}\right)\varphi\left(\frac{x^2}{h^2(n)}\right)$$ 
for some $\varphi \in C^{\infty}_c([0,\infty))$ s.t. $\chi_{[0,1]} \leq \varphi \leq \chi_{[0,2]}$. We say a sequence of functions $\{f_n\}$ converges \emph{boundedly and point-wise} to $f$ (in short $\textrm{bp-lim}_{n\rightarrow \infty} f_n=  f$) if it holds that $\sup_n \lvert f_n \rvert<\infty$ and $\lim_{n\rightarrow \infty} f_n(x) = f(x)$ for all $x\in S$. Clearly, for the sequence $\{f_n\}_{n\in \N}$ as defined above, it holds that 
$$\textrm{bp-lim}_{n\rightarrow \infty} f_n= \chi_{[0, \infty)\times S}.$$ 
Let $\Delta$ be the point at infinity of $[0,\infty) \times S$ and denote by $\big([0,\infty) \times S\big)^{\Delta}$ the one-point-compactification of $[0,\infty) \times S$. For a generator $ \mathcal{H}^0$ we define its extension $( \mathcal{H}^0)^{\Delta}$ acting on $C(([0,\infty) \times S)^{\Delta})$ by
$$ ( \mathcal{H}^0)^{\Delta} f = f(\Delta) +  \mathcal{H}^0(f- f(\Delta))$$
for $f\in C\big(([0,\infty) \times S\big)^{\Delta})$. Then, for $g_n:= ( \mathcal{H}^0)^{\Delta}f_n$, it holds that $g_n \rightarrow 0$ point-wise. The crucial step, where the growth conditions~\eqref{eq:sigma_growth} and~\eqref{eq:xb_growth} are needed, is to ensure that $\inf_n \inf_{(t,x)} g_n(t,x) > -\infty$. By the definition of $g_n:= ( \mathcal{H}^0)^{\Delta}f_n$ and the properties of $\varphi$, we find that:
\begin{align*}
    \inf_{(t,x)} g_n(t,x)&\geq  \inf_{(t,x)} \frac{1}{n}\underbrace{\varphi'\left(\frac{t}{n}\right)}_{\leq 0} \underbrace{\varphi\left(\frac{x^2}{h^2(n)}\right)}_{\geq 0} + \inf_{(t,x)} \Bigg[\frac{2}{h^2(n)} x\cdot B(t,x) \underbrace{\varphi\left(\frac{t}{n}\right)}_{\geq 0} \underbrace{\varphi'\left(\frac{x^2}{h^2(n)}\right)}_{\leq 0}\\ & \quad +\frac{1}{h^2(n)} \sigma^2(t,x)\Bigg( \underbrace{\varphi'\left(\frac{x^2}{h^2(n)}\right)}_{\leq 0} \underbrace{\varphi\left(\frac{t}{n}\right)}_{\geq 0} + \frac{2x^2}{h^4(n)} \varphi''\left(\frac{x^2}{h^2(n)}\right)\underbrace{\varphi\left(\frac{t}{n}\right)}_{\geq 0}\Bigg)\Bigg].
\end{align*}

Considering the first term only, the 
infimum in $(t,x)$ is attained for some $x\in [0,h(n)]$ (resp. $x\in [-h(n), -h(n)]$ if $S=\R$ or $x\in [0, \min\{h(n), 1\}]$ if $S=[0,1]$) and $t \in [n, 2n]$. Then $$\inf_n \inf_{(t,x)} \frac{1}{n}\underbrace{\varphi'\left(\frac{t}{n}\right)}_{\leq 0} \underbrace{\varphi\left(\frac{x^2}{h^2(n)}\right)}_{\geq 0} >-\infty,
$$ and we are only left to study the second term. We first treat the cases $S=\R_{\geq 0}$ and $S=\R$ for which it follows that the infimum of the second term is attained for $t=n$ and $x= \alpha h(n) $ for some $\alpha\in[1,\sqrt{2}]$ (resp. $\alpha\in[-\sqrt{2}, -1]\cup[1,\sqrt{2}]$ if $S=\R$) fixed. Plugging this in and using~\eqref{eq:sigma_growth},~\eqref{eq:xb_growth}, we obtain 
\begin{align*}
&\inf_{n} \Bigg[ \frac{2\alpha h(n) B(n,\alpha h(n))}{h^2(n)} \underbrace{\varphi\left(1\right)}_{\geq 0} \underbrace{\varphi'\left(\alpha^2\right)}_{\leq 0} + \frac{\sigma^2(n,\alpha h(n))}{h^2(n)} \underbrace{\varphi\left(1\right)}_{\geq 0}\Bigg( \underbrace{\varphi'\left(\alpha^2\right)}_{\leq 0}  + 2\alpha^2 \varphi''\left(\alpha^2\right)\Bigg)\Bigg]\\ & \quad \geq \inf_{n}\Bigg[ \frac{2C^B\cdot(1 + h^2(n))}{h^2(n)} \underbrace{\varphi\left(1\right)}_{\geq 0} \underbrace{\varphi'\left(\alpha^2\right)}_{\leq 0} \\& \qquad+ \frac{C^{\sigma}\cdot(1 + h^2(n))}{h^2(n)} \Bigg( \underbrace{\varphi'\left(\alpha^2\right)}_{\leq 0} \underbrace{\varphi\left(1\right)}_{\geq 0} + 2\alpha^2 \varphi''\left(\alpha^2\right)\underbrace{\varphi\left( 1 \right)}_{\geq 0}\Bigg)\Bigg] \\& \quad > -\infty.
\end{align*}
For the case $S=[0,1]$ the reasoning above is exactly the same, except $x= \min\{\alpha h(n) , 1\}$ for some $\alpha \in [1,\sqrt{2}]$.\\

Secondly, the path-wise uniqueness follows by~\cite[Theorem 5.3.8 and Remark 5.3.9]{EK}. 
\end{proof}

\subsection{Proof of Lemma~\ref{lem:new_ass_maps}}\label{app:new_ass_maps}
\begin{proof}[Proof of Lemma~\ref{lem:new_ass_maps}]
We begin with the proof for the case of Assumption~\ref{as:A} and later turn to Assumption~\ref{as:B}.

Our first step is to show that $\big\lvert \E[Z_t]\big\rvert \leq (\lvert\E[Z_0]\rvert + b_0t)\exp((b_0 +\beta_0)t)$ for all $t\geq 0$. To do so, we can proceed similarly as in~\cite[Proof of Theorem 2.10]{Cuchiero_2012}. Let us introduce a sequence of stopping times $(\tau_n)_{n\in \N}$ with $\tau_n = \inf \{ t >0 \lvert \, \lvert Z_t \rvert \geq n \}$ and note that for any $n \in \N$
\begin{align*}
\big\lvert \E[Z_{t\wedge \tau_n}] \big\rvert &\leq\big\lvert \E[Z_0] \big\rvert + \Big \lvert \E\Big[\int_0^{t\wedge \tau_n} b(\E[\Z_s]) + \beta(\E[\Z_s])Z_s ds\Big] \Big \rvert + \big\lvert \E[M^{(1)}_{t\wedge \tau_n}]\big\rvert \\ &\leq \big\lvert \E[Z_0] \big\rvert + \int_0^{t} \big\lvert b(\E[\Z_{s\wedge \tau_n}])\big\rvert + \big\lvert\beta(\E[\Z_{s\wedge \tau_n}])\big\rvert \big\lvert\E[Z_{s\wedge \tau_n}]\big\rvert ds\\ & \leq \big\lvert \E[Z_0] \big\rvert  + \int_0^t b_0(1+ \big\lvert\E[Z_{s\wedge \tau_n}]\big\rvert) + \beta_0 \big\lvert\E[Z_{s\wedge \tau_n}]\big\rvert ds,
\end{align*}
where we used that $(M^{(1)}_{t\wedge \tau_n})_{t\geq0}$ is a martingale. Using Gronwall's inequality, i.e. the fact that if $$X_t \leq \alpha(t)+ \int_0^t \beta(s) X_s ds \Rightarrow X_t \leq \alpha(t)\exp\left(\int_0^t \beta(s)ds\right)$$
for $\alpha: [0,\infty) \rightarrow \R$ continuous and increasing and $\beta: [0, \infty) \rightarrow [0, \infty)$, we find that
$$\big\lvert \E[Z_{t\wedge \tau_n}]\big\rvert \leq \left(\big\lvert\E[Z_0]\big\rvert + b_0t\right)\exp\left((b_0 +\beta_0)t\right).$$
Since $(Z_t)_{t\geq0}$ is continuous and adapted, it is locally bounded and $\lim_{n\rightarrow \infty} \tau_n = \infty$ and 
and it thus follows by Fatou's lemma that for all $t\geq 0$
$$ \big\lvert \E[Z_t]\big\rvert \leq \liminf_{n\rightarrow \infty} \big \lvert \E[Z_{t\wedge\tau_n}]\big \rvert \leq \left(\big\lvert\E[Z_0]\big\rvert + b_0t\right)\exp\left((b_0 +\beta_0)t\right)=:\zeta(t),$$
where we abbreviate this upper bound by $\zeta$ for future convenience. 

We now move on to the case $k\geq 2$ and aim to show that \eqref{eq:true_mart} holds. Note that for $2\leq k\leq N$ we find
\begin{align*}
&\int_0^{t} \E\Big[\Big\lvert \big(c(\E[\Z_s])+ \gamma(\E[\Z_s])Z_s + \Gamma(\E[\Z_s])Z_s^2\big) Z_s^{2k-2}\Big\rvert\Big] ds \\& \quad \leq  \int_0^{t} \big\lvert c(\E[\Z_s])\big \rvert\cdot \E[\lvert Z_s\rvert^{2k-2}] + \big\lvert\gamma(\E[\Z_s])\big \rvert \cdot \E[\lvert Z_s\rvert^{2k-1}] + \big \lvert \Gamma(\E[\Z_s])\big \rvert \cdot\E[\lvert Z_s\rvert^{2k}] ds\\& \quad \leq \int_0^{t} c_0\big(1 + \|\E[\Z_s]\|^2_N + f_c(\lvert\E[Z_s]\rvert)\big)\cdot \E[\lvert Z_s\rvert^{2k-2}]ds\\ & \qquad + \int_0^t\gamma_0\big(1+  \|\E[\lvert \Z_s\rvert]\|_N + f_{\gamma}(\lvert\E[Z_s]\rvert)\big) \cdot   \E[\lvert Z_s\rvert^{2k-1}] + \Gamma_0 \cdot\E[\lvert Z_s\rvert^{2k}] ds
\end{align*}
Let us recall that $\lvert\E[Z_t]\rvert \leq \zeta(t)$ for all $t\geq 0$ and $\sup_{s\in[0,t]} f_c(\zeta(s)) = f_c(\zeta(t)) < \infty$ and that the same holds for $f_{\gamma}$ by the definition of being a continuous increasing function. 
Moreover, let us note that for all $i \in \{1, \dots, N\}$
\begin{align*}
\lvert \E[Z^i_t]\rvert^{\frac{2}{i}} = \lvert \E[Z^i_t]\rvert^{\frac{2N}{iN}} \leq \E[\lvert Z_t\rvert^{2N}]^{\frac{1}{N}} \leq 1+ \E[\lvert Z_t\rvert^{2N}]
\end{align*}
by Jensen's inequality and the fact that $x \leq 1 + x^N$ for all $x\in \R$. Hence, 
\begin{align*}
\|\E[\Z_t]\|_N \leq \sqrt{N + N\E[\lvert Z_t\rvert^{2N}]} \leq 1+ N+ N\E[\lvert Z_t\rvert^{2N}].
\end{align*}
Using the same reasoning, it holds for all $k \in \{1, \dots, N\}$ that 
\begin{align*}
\E[\lvert Z_t\rvert^{2k-2}]\leq 1 + \E[\lvert Z_t\rvert^{2N}],\quad \E[\lvert Z_t\rvert^{2k-1}]\leq 1 + \E[\lvert Z_t\rvert^{2N}]\quad \mathrm{and} \quad \E[\lvert Z_t\rvert^{2k}]\leq 1 + \E[\lvert Z_t\rvert^{2N}].
\end{align*}
Hence it is sufficient to show that for all $t\geq 0$,  $\sup_{s\in [0,t]}  \E[\lvert Z_s \rvert^{2N}]< \infty$ in order for \eqref{eq:true_mart} to hold for all $k\in \{1, \dots , N\}$. In the following we will denote by $\lesssim$ an inequality up to a multiplicative constant and if we write e.g. $\lesssim_{N}$ we mean that this multiplicative constant depends on $N$. 
We apply the same reasoning as above to show that 
\begin{align*}
\lvert \E[Z^i_t]\rvert^{\frac{2}{i}} = \lvert \E[Z^i_t]\rvert^{\frac{2N}{iN}} \leq \E[\lvert Z_t\rvert^{N}]^{\frac{2}{N}}\end{align*}
and hence, $$\|\E[\Z_t]\|_N \lesssim_N \E[\lvert Z_t\rvert^N]^{\frac{1}{N}}.$$
Again, we denote by $(\tau_n)_{n\in \N}$ the sequence of stopping times with $\tau_n :=\inf\{t>0 \lvert \, \lvert Z_t\rvert \geq n\}$. We are now ready to show that 
\begin{align*}
\E[\lvert &Z_{t\wedge \tau_n} \rvert^{2N}]  \leq \E\Bigg[ \Big \lvert Z_0  + \int_0^{t\wedge \tau_n} (b(\E[\Z_r])+ \beta(\E[\Z_r])Z_r)dr\\ & \quad + \int_0^{t\wedge\tau_n} \sqrt{c(\E[\Z_r])+ \gamma(\E[\Z_r])Z_r +\Gamma(\E[\Z_r]) Z_r^2} dW_r\Big \rvert^{2N}\Bigg] \\ & \lesssim_N \E[\lvert Z_0\rvert^{2N}] +  t^{2N-1}\E\Bigg[\int_0^{t} \lvert b(\E[\Z_{r\wedge \tau_n}])+ \beta(\E[\Z_{r\wedge \tau_n}])Z_{r\wedge \tau_n} \rvert^{2N} dr \Bigg] \\ & \quad + \E\Bigg [\int_0^{t} \lvert c(\E[\Z_{r\wedge \tau_n}])+ \gamma(\E[\Z_{r\wedge \tau_n}])Z_{r\wedge \tau_n} + \Gamma(\E[\Z_{r\wedge \tau_n}]) Z^2_{r\wedge \tau_n}\rvert^N dr  \Bigg] \\& \lesssim_N \E[\lvert Z_0\rvert^{2N}] + t^{2N-1}\int_0^t \lvert b(\E[\Z_{r\wedge \tau_n}]) \rvert^{2N} + (\beta_0)^{2N}\cdot  \E[\lvert Z_{r\wedge \tau_n} \rvert^{2N}] dr \\ & \quad + \int_0^t \lvert c(\E[\Z_{r\wedge \tau_n}])\rvert^N + \lvert\gamma(\E[\Z_{r\wedge \tau_n}])\rvert^N \E[\lvert Z_{r\wedge \tau_n}\rvert^N]+ (\Gamma_0)^N \E[\lvert Z_{r\wedge \tau_n}\rvert^{2N}]dr  \\ &\lesssim_N \E[\lvert Z_0\rvert^{2N}] + t^{2N-1}\int_0^t  (b_0)^{2N}(1 + \lvert\E[Z_{r\wedge \tau_n}]\rvert^{2N})+ (\beta_0)^{2N}\E[\lvert Z_{r\wedge \tau_n} \rvert^{2N}] dr \\ & \quad + \int_0^t \Big((c_0)^N \big(1+\E[\lvert Z_{r\wedge \tau_n}\rvert^{N}]^2 + f_c(\lvert \E[Z_{r\wedge \tau_n}]\rvert)^N\big)\\& \qquad+ (\gamma_0)^N\big(1 +\E[\lvert Z_r\rvert^{N}] + f_{\gamma}(\lvert\E[Z_{r\wedge \tau_n}]\rvert)^N\big)\E[\lvert Z_{r\wedge \tau_n}\rvert^N]+(\Gamma_0)^N\E[\lvert Z_{r\wedge \tau_n}\rvert^{2N}]\Big)dr  \\ & \leq \E[\lvert Z_0\rvert^{2N}] + ( b_0 t)^{2N} +  \left((b_0)^{2N}+ (\beta_0)^{2N}\right)t^{2N-1} \int_0^t \E\left[\lvert Z_{r\wedge \tau_n} \rvert^{2N}\right] dr \\&
\qquad+(c_0)^N(1+ f_c(\zeta(t))^N)t + (\gamma_0)^N(1+f_{\gamma}(\zeta(t))^N)t\\&
\quad + \int_0^t \Big( (c_0)^N+ ({\gamma}_0)^N\big(2+f_{\gamma}(\zeta(r))\big)+ (\Gamma_0)^N\Big) \E\left[\lvert Z_{r\wedge \tau_n}\rvert^{2N}\right]dr 
\end{align*}
where we have used the BDG inequality to obtain the second inequality. Moreover, we used Jensen's inequality and the fact that $(a+b)^n \leq 2^{n-1}(\lvert a\rvert^n +\lvert b \rvert^n)$ several times. Again, using Gronwall's inequality, it follows that 
\begin{align}\label{eq:E_sup_ineq_A1}
&\E[\lvert Z_{t\wedge \tau_n} \rvert^{2N}] \leq C_N \Big(\E[\lvert Z_0\rvert^{2N}]+ ( b_0t )^{2N}+  (c_0)^N(1+ f_c(\zeta(t))^N)t + (\gamma_0)^N(1+f_{\gamma}(\zeta(t))^N)t\Big)\nonumber \\&\qquad \cdot \exp{\Bigg(C_N\Big((b_0+\beta_0)^{2N}t^{2N-1}+(c_0)^N+ ({\gamma}_0)^N\big(2+f_{\gamma}(\zeta(t))\big)+ (\Gamma_0)^N \Big) t \Bigg)}=:\zeta_1^{(2N)}(t)
\end{align}
for some constant $C_N\in\R_{\geq0}$. By Fatou's lemma, $\E[ \lvert Z_t\rvert^{2N}] \leq \zeta_1^{(2N)}(t)$ and hence also $\sup_{s\in[0,t]} \E[\lvert Z_s \rvert^{2N}] \leq \zeta_1^{(2N)}(t)$. Therefore, we have shown that \eqref{eq:true_mart} holds for all $k\in \{1, \dots, N\}$ and $t\geq 0$. Note, for the above proof, we adapted the reasoning of the linear-growth conditions on coefficients of Itô-SDEs (see e.g.~\cite[Lemma 16.1.4]{Cohen_Eliott}) and translated them to the maps $b, \beta, c, \gamma, \Gamma: \R^N \rightarrow \R$. 

We have now established the case of the set of Assumptions~\ref{as:A}. The reasoning is very similar to the previous case. In particular, it is enough to show that $\sup_{s\in[0,t]} \E[\lvert Z_s\rvert^{2N}]< \infty$ for all $t\in [0,\infty)$ for \eqref{eq:true_mart} to hold. To this end, we introduce the sequence of stopping times $(\tau_n)_{n\in \N}$ with $\tau_n:=\inf\{ t> 0 \lvert \,  \lvert Z_t\rvert \geq n\}$. Using similar arguments as before, we find 
\begin{align*}
\E[\lvert Z_{t\wedge \tau_n} \rvert^{2N}] & \leq \E\Bigg[ \Big \lvert Z_0  + \int_0^{t\wedge \tau_n} (b(\E[\Z_r])+ \beta(\E[\Z_r])Z_r)dr\\ & \quad + \int_0^{t\wedge\tau_n} \sqrt{c(\E[\Z_r])+ \gamma(\E[\Z_r])Z_r +\Gamma(\E[\Z_r]) Z_r^2} dW_r\Big \rvert^{2N}\Bigg] \\ & \lesssim_N \E[\lvert Z_0\rvert^{2N}] + t^{2N-1}\int_0^t \lvert b(\E[\Z_{r\wedge \tau_n}]) \rvert^{2N} + (\beta_0)^{2N}\cdot  \E[\lvert Z_{r\wedge \tau_n} \rvert^{2N}] dr \\ & \quad + \int_0^t \lvert c(\E[\Z_{r\wedge \tau_n}])\rvert^N + \lvert\gamma(\E[\Z_{r\wedge \tau_n}])\rvert^N \E[\lvert Z_{r\wedge \tau_n}\rvert^N]+ (\Gamma_0)^N \E[\lvert Z_{r\wedge \tau_n}\rvert^{2N}]dr  \\ &\lesssim_N \E[\lvert Z_0\rvert^{2N}] + t^{2N-1}\int_0^t  (b_0)^{2N}(1 + \E[\lvert Z_{r\wedge \tau_n}\rvert^N]^{2})+ (\beta_0)^{2N}\E[\lvert Z_{r\wedge \tau_n} \rvert^{2N}] dr \\ & \quad + \int_0^t \Big((c_0)^N \big(1+\E[\lvert Z_{r\wedge \tau_n}\rvert^{N}]^2 \big)+ (\gamma_0)^N\big(1 +\E[\lvert Z_r\rvert^{N}] \big)\E[\lvert Z_{r\wedge \tau_n}\rvert^N]\\ & \qquad +(\Gamma_0)^N\E[\lvert Z_{r\wedge \tau_n}\rvert^{2N}]\Big)dr  \\ & \leq \E[\lvert Z_0\rvert^{2N}] + ( b_0 t)^{2N} +  \left((b_0)^{2N}+ (\beta_0)^{2N}\right)t^{2N-1} \int_0^t \E\left[\lvert Z_{r\wedge \tau_n} \rvert^{2N}\right] dr \\&
\qquad+(c_0)^Nt + (\gamma_0)^Nt + \int_0^t \Big( (c_0)^N+ ({\gamma}_0)^N+ (\Gamma_0)^N\Big) \E\left[\lvert Z_{r\wedge \tau_n}\rvert^{2N}\right]dr 
\end{align*}
and hence, again by Gronwall's inequality, it holds
\begin{align}\label{eq:E_sup_ineq_A2}
&\E[\lvert Z_{t\wedge \tau_n} \rvert^{2N}] \leq C_N\cdot \Big(\E[\lvert Z_0\rvert^{2N}]+ ( b_0t )^{2N}+  (c_0)^Nt + (\gamma_0)^N
t\Big)\nonumber \\&\qquad \cdot \exp{\Bigg(C'_N\cdot\Big((b_0+\beta_0)^{2N}t^{2N-1}+(c_0)^N+ ({\gamma}_0)^N+ (\Gamma_0)^N \Big) t \Bigg)}=:\zeta_2^{(2N)}(t)
\end{align}
for some constant $C'_N\in \R_{\geq 0}$. By Fatou's lemma, it holds $\E[\lvert Z_t\rvert^{2N}] \leq \zeta_2^{(2N)}(t)$ for all $t\in [0, \infty)$. Since this implies $\sup_{s\in[0,t]}\E[\lvert Z_s\rvert^{2N}] \leq \zeta_2^{(2N)}(t)$ for all $t\in [0, \infty)$ and we have shown the local martingale property also for the set of Assumptions~\ref{as:B}.
\end{proof}

\subsection{Proof of Proposition~\ref{prop:global_sol_ode}}\label{app:global_sol_ode}

\begin{proof}[Proof of Proposition~\ref{prop:global_sol_ode}]
The existence of a unique solution up to its maximal lifetime is assured by the Lipschitz condition. To see the global existence for $A= \R^{N+1}$, we apply Gronwall's inequality. Clearly, the zeroth component is constant, i.e. $\bar{z}_0(t)\equiv \bar{z}_0(0)= 1$. 

Let us first treat the set of Assumptions~\ref{as:A} to prove the statement, and later adapt the proof for the set of Assumptions~\ref{as:B}. To this end, we start by studying the first component and denote by $(t_n)_{n\in \N} $ a (deterministic) sequence of time-points such that $t_n := \inf\{t>0 \lvert \, \| z(t)\|^{2N}_N \geq n\}$. Recall that we have already established that a solution exists at least on some interval $[0, t^*)$ where $t^* = \inf\{ t>0 \lvert \,  \|z(t)\|^{2N}_N = \infty\}$. For global existence, we would like to show that $t^*= \infty$, for which a sufficient condition is that $\lim_{n\rightarrow \infty} t_n \rightarrow \infty$. Then, considering the first component, we obtain
\begin{align*}
    \sup_{s\in[0,t\wedge t_n]} \lvert  z_1(s) \rvert &= \sup_{s\in[0,t\wedge t_n]}\Big\lvert z_1(0)+ \int_0^s {\beta}(z(r))z_1(r)+ b(z(r)) dr \Big\rvert \\ & = \lvert Z_0 \rvert + \int_0^{t\wedge t_n} \lvert\beta(z(r))\rvert\lvert z_1(r)\rvert+ \lvert b(z(r))\rvert dr \\ & \leq  \lvert z_1(0) \rvert + \int_0^{t\wedge t_n} \beta_0 \lvert z_1(r\wedge t_n) \rvert + b_0 + K_b\lvert z_1(r\wedge t_n) \rvert dr \\ &\leq \lvert Z_0 \rvert + b_0({t\wedge t_n}) + \int_0^{t\wedge t_n} (\beta_0 + K_b) \sup_{u\in [0,r\wedge t_n]} \lvert z_1(u) \rvert dr
\end{align*}
and hence by Gronwall's inequality, we find
$$ \sup_{s\in [0,t\wedge t_n]} \lvert z_1(s)\rvert \leq \left(\lvert z_1(0) \rvert + b_0 ({t\wedge t_n})\right)\cdot \exp\left( (\beta_0 + K_b)({t\wedge t_n})  \right) =: h_1(t\wedge t_n).$$

We will not treat the following components one-by-one, but instead study $\|z(t)\|_N $.
Note that the first component plays a special role in this set of assumptions as it only depends on $\beta, b$ which themselves are bounded by the first component, while the other maps' bounds depend on all components.
Clearly, if $\|z(t)\|^{2N}_N < \infty$ for all $t\in [0, \infty)$ this implies  $\lvert Z_1(t)\rvert, \lvert z_2(t)\rvert, \dots, \lvert z_N(t)\rvert < \infty$ for all $t\in [0, \infty)$. In preparation to that, let us note that for all $k \in \{1, \dots N\}$ and $x,y \in \R$ it holds that $\lvert x\rvert^k + \lvert y \rvert^k\leq (\lvert x \rvert + \lvert y \rvert)^k$. Hence, it also holds that for all $k \in \{1, \dots, N\}$ and $t\geq 0$ that 
\begin{align*}
\lvert z_k(t) \rvert^2 = \left(\lvert z_k(t) \rvert^{\frac{2}{k}}\right)^k \leq \left(\lvert z_k(t) \rvert^{\frac{2}{k}}\right)^k + \Big(\sum_{\substack{j=1\\ j\neq k}}^N \lvert z_j(t) \rvert^{\frac{2}{j}}\Big)^k \leq \Big(\sum_{j=1}^N \lvert z_j(t) \rvert^{\frac{2}{j}}\Big)^k = \Big(\|z(t)\|_N\Big)^{2k}.
\end{align*}
We are now ready to study $\|z(t\wedge t_n)\|_N^{2N}$. 
\begingroup
\allowdisplaybreaks
\begin{align*}
&\|z(t\wedge t_n)\|_N^{2N} \lesssim_N \sum_{j=1}^N \lvert z_j(t)\rvert^{\frac{2N}{j}} \\ & \quad \leq \sum_{j=1}^N \Bigg\lvert z_j(0)+ \int_0^{t\wedge t_n} \left(j\cdot\beta(z(s)) + \frac{j(j-1)}{2} \Gamma(z(s))\right) \bar{z}_j(s) ds \\ & \qquad + \int_0^{t\wedge t_n} \left(j\cdot b(z(s)) + \frac{j(j-1)}{2} \gamma(z(s)) \right) \bar{z}_{j-1}(s) ds \\ & \qquad + \mathbf{1}_{\{j>1\}}\int_0^{t\wedge t_n}   \frac{j(j-1)}{2} c(z(s))\bar{z}_{j-2}(s) ds\Bigg\rvert^{\frac{2N}{j}} \\ & \quad\lesssim_N \sum_{j=1}^N \lvert z_j(0)\rvert^{\frac{2N}{j}}+ (t\wedge t_n)^{\frac{2N}{j}-1}\int_0^{t\wedge t_n}  \left(\lvert \beta(z(s)) \rvert^{\frac{2N}{j}}+ \mathbf{1}_{\{j>1\}} \lvert \Gamma(z(s))\rvert^{\frac{2N}{j}}\right)\lvert z_j(s)\rvert^{\frac{2N}{j}} ds \\ &\qquad + (t\wedge t_n)^{\frac{2N}{j}-1}\int_0^{t\wedge t_n}  \left( \lvert b(z(s))\rvert^{\frac{2N}{j}} + \mathbf{1}_{\{j>1\}} \lvert\gamma(z(s))\rvert^{\frac{2N}{j}}\right)\lvert \bar{z}_{j-1}(s)\rvert^{\frac{2N}{j}} ds \\ &\qquad + \mathbf{1}_{\{j>1\}} (t\wedge t_n)^{\frac{2N}{j}-1}\int_0^{t\wedge t_n}  \lvert c(z(s))\rvert^{\frac{2N}{j}} \lvert \bar{z}_{j-2}(s)\rvert^{\frac{2N}{j}} ds\\ & \quad \lesssim_N \|z(0)\|_N^{2N}+ (1+ \beta_0^{2N}+ \Gamma_0^{2N})(1+ t\wedge t_n)^{2N}\int_0^{t\wedge t_n}\|z(s)\|^{2N}_N ds\\ & \qquad + \big(b_0(1+ h_1(t))\big)^{2N} \sum_{j=1}^N (t\wedge t_n)^{\frac{2N}{j}-1}\int_0^{t\wedge t_n} \lvert \bar{z}_{j-1}(s)\rvert^{\frac{2N}{j}}ds \\ & \qquad + \sum_{j=2}^N (t\wedge t_n)^{\frac{2N}{j}-1}\int_0^{t\wedge t_n}  \big(\gamma_0(1+ f_{\gamma}(\lvert z_1(s)\rvert) + \|z(s)\|_N)\big)^{\frac{2N}{j}} \lvert \bar{z}_{j-1}(s)\rvert^{\frac{2N}{j}} ds \\ &\qquad + \sum_{j=2}^N (t\wedge t_n)^{\frac{2N}{j}-1}\int_0^{t\wedge t_n}  \big(c_0(1+ f_{c}(\lvert z_1(s)\rvert) + \|z(s)\|_N^2)\big)^{\frac{2N}{j}} \lvert \bar{z}_{j-2}(s)\rvert^{\frac{2N}{j}} ds\\ &\quad \lesssim_N \|z(0)\|_N^{2N}+ (1+ \beta_0^{2N}+ \Gamma_0^{2N})(1+ t\wedge t_n)^{2N}\int_0^{t\wedge t_n}  \|z(s)\|^{2N}_N ds\\ & \qquad + \big(b_0(1+ h_1(t\wedge t_n))\big)^{2N} \sum_{j=1}^N (t\wedge t_n)^{\frac{2N}{j}-1}\int_0^{t\wedge t_n} \left(1+ \| z(s)\|_N^j\right)\rvert^{\frac{2N}{j}}ds \\ & \qquad + \sum_{j=2}^N (t\wedge t_n)^{\frac{2N}{j}-1}\int_0^{t\wedge t_n}  \big(\gamma_0(1+ f_{\gamma}(\lvert z_1(s)\rvert) + \|z(s)\|_N)\big)^{\frac{2N}{j}} (1+\| z(s)\|_N^{j-1})^{\frac{2N}{j}} ds\\ &\qquad +\sum_{j=2}^N (t\wedge t_n)^{\frac{2N}{j}-1}\int_0^{t\wedge t_n} \big(c_0(1+ f_{c}(\lvert z_1(s)\rvert) + \|z(s)\|^2_N)\big)^{\frac{2N}{j}} (1+\| z(s)\|_N^{j-2})^{\frac{2N}{j}} ds\\ &\quad \lesssim_N \|z(0)\|_N^{2N} + (1+ \beta_0^{2N}+ \Gamma_0^{2N})(1+ t\wedge t_n)^{2N}\int_0^{t\wedge t_n} \|z(s)\|^{2N}_N ds \\ & \qquad + (1+ t\wedge t_n)^{2N}\Big[\Big(b_0 + \gamma_0+ c_0\Big)\Big(1+ h_1(t\wedge t_n) +f_{\gamma}\big(h_1(t\wedge t_n)\big)+ f_c\big(h_1(t \wedge t_n)\big)\Big)\Big]^{2N} \\& \qquad \cdot \int_0^{t\wedge t_n}  (2 + \|z(s)\|_N^{2N}) ds\\ & \qquad +  \sum_{j=2}^N \Big[ (1+\gamma_0^{2N})  (t\wedge t_n)^{\frac{2N}{j}-1}\int_0^{t\wedge t_n}(\|z(s)\|_N^{\frac{2N}{j}} + \|z(s)\|_N^{2N}) ds \\ & \qquad \qquad + (1+ c_0^{2N})(t\wedge t_n)^{\frac{2N}{j}-1}\int_0^{t\wedge t_n} (\|z(s)\|_N^{\frac{4N}{j}}+ \|z(s)\|_N^{2N}) ds \Big] \\ &\quad \lesssim_N \|z(0)\|_N^{2N} + \Big[\Big(1+ \beta_0+ \Gamma_0 +b_0 + \gamma_0+ c_0\Big)\\& \qquad \cdot \Big(1+ h_1(t\wedge t_n) +f_{\gamma}\big(h_1(t\wedge t_n)\big)+ f_c\big(h_1(t\wedge t_n)\big)\Big)(1+ t\wedge t_n)\Big]^{2N} \\ & \qquad \cdot \Big[ t+ \int_0^t \|z({s\wedge t_n} )\|_N^{2N} ds \Big]
\end{align*}
\endgroup
Hence, by an application of Gronwall's inequality, we can show that $\|z(t\wedge t_n)\|_N^{2N}< h_{2N}(t\wedge t_n)$ for all $n\in \N$ and $t\geq 0$ and a monotonically increasing continuous function $h_{2N}: [0, \infty) \rightarrow [0, \infty)$. Now, let's assume $\lim_{n\rightarrow \infty} t_n = t^* < \infty$. However,  
$$\lim_{n\rightarrow \infty}\|z(t\wedge t_n)\|_N^{2N}\leq \lim_{n\rightarrow \infty}h_{2N}(t\wedge t_n) = h_{2N}(t^*)< \infty$$ which is a contradiction because by definition $\lim_{n\rightarrow \infty}\|z(t\wedge t_n)\|_N^{2N} \rightarrow \infty$. Hence, we have shown that indeed a global solution exists under the set of Assumptions~\ref{as:A}. 

The proof using the set of Assumptions~\ref{as:B} is structurally very similar. Since the first component does not play a special role in this set of assumptions, we can start directly by treating $\| z(t) \|_N^{2N}$, where we introduce again the sequence $(t_n)_{n\in \N}$ with $$t_n := \inf\{t>0 \,\lvert \, \|z(t)\|_N^{2N} \geq n\}.$$ Using the same arguments as before, we can show 
\begingroup
\allowdisplaybreaks
\begin{align*}
&\|z(t\wedge t_n)\|_N^{2N} \lesssim_N \sum_{j=1}^N \lvert z_j(t)\rvert^{\frac{2N}{j}} \\ & \quad\lesssim_N \sum_{j=1}^N \lvert z_j(0)\rvert^{\frac{2N}{j}}+ (t\wedge t_n)^{\frac{2N}{j}-1}\int_0^{t\wedge t_n}  \left(\lvert \beta(z(s)) \rvert^{\frac{2N}{j}}+ \mathbf{1}_{\{j>1\}} \lvert \Gamma(z(s))\rvert^{\frac{2N}{j}}\right)\lvert z_j(s)\rvert^{\frac{2N}{j}} ds \\ &\qquad + (t\wedge t_n)^{\frac{2N}{j}-1}\int_0^{t\wedge t_n}  \left( \lvert b(z(s))\rvert^{\frac{2N}{j}} + \mathbf{1}_{\{j>1\}} \lvert\gamma(z(s))\rvert^{\frac{2N}{j}}\right)\lvert \bar{z}_{j-1}(s)\rvert^{\frac{2N}{j}} ds \\ &\qquad + \mathbf{1}_{\{j>1\}} (t\wedge t_n)^{\frac{2N}{j}-1}\int_0^{t\wedge t_n}  \lvert c(z(s))\rvert^{\frac{2N}{j}} \lvert \bar{z}_{j-2}(s)\rvert^{\frac{2N}{j}} ds\\ & \quad \lesssim_N \|z(0)\|_N^{2N}+ (1+ \beta_0^{2N}+ \Gamma_0^{2N})(1+ t\wedge t_n)^{2N}\int_0^{t\wedge t_n}\|z(s)\|^{2N}_N ds\\ & \qquad + \sum_{j=1}^N (t\wedge t_n)^{\frac{2N}{j}-1}\int_0^{t\wedge t_n} \big(b_0(1+ \|z(s)\|_N)\big)^{\frac{2N}{j}} \lvert \bar{z}_{j-1}(s)\rvert^{\frac{2N}{j}}ds \\ & \qquad + \sum_{j=2}^N (t\wedge t_n)^{\frac{2N}{j}-1}\int_0^{t\wedge t_n}  \big(\gamma_0(1+ \|z(s)\|_N)\big)^{\frac{2N}{j}} \lvert \bar{z}_{j-1}(s)\rvert^{\frac{2N}{j}} ds \\ &\qquad + \sum_{j=2}^N (t\wedge t_n)^{\frac{2N}{j}-1}\int_0^{t\wedge t_n}  \big(c_0(1+ \|z(s)\|_N^2)\big)^{\frac{2N}{j}} \lvert \bar{z}_{j-2}(s)\rvert^{\frac{2N}{j}} ds\\ &\quad \lesssim_N \|z(0)\|_N^{2N}+ (1+ \beta_0^{2N}+ \Gamma_0^{2N})(1+ t\wedge t_n)^{2N}\int_0^{t\wedge t_n}  \|z(s)\|^{2N}_N ds \\ & \qquad + \sum_{j=2}^N (t\wedge t_n)^{\frac{2N}{j}-1}\int_0^{t\wedge t_n}  \big((b_0+\gamma_0)(1+ \|z(s)\|_N)\big)^{\frac{2N}{j}} (1+\| z(s)\|_N^{j-1})^{\frac{2N}{j}} ds\\ &\qquad +\sum_{j=2}^N (t\wedge t_n)^{\frac{2N}{j}-1}\int_0^{t\wedge t_n} \big(c_0(1 + \|z(s)\|^2_N)\big)^{\frac{2N}{j}} (1+\| z(s)\|_N^{j-2})^{\frac{2N}{j}} ds\\ &\quad \lesssim_N \|z(0)\|_N^{2N} + (1+ \beta_0^{2N}+ \Gamma_0^{2N})(1+ t\wedge t_n)^{2N}\int_0^{t\wedge t_n} \|z(s)\|^{2N}_N ds \\ & \qquad + (1+ t\wedge t_n)^{2N}\Big(b_0 + \gamma_0+ c_0\Big)^{2N} \int_0^{t\wedge t_n}  (1 + \|z(s)\|_N^{2N}) ds\\ & \qquad +  \sum_{j=2}^N \Big[ (1+\gamma_0^{2N}+ b_0^{2N})  (t\wedge t_n)^{\frac{2N}{j}-1}\int_0^{t\wedge t_n}(\|z(s)\|_N^{\frac{2N}{j}} + \|z(s)\|_N^{2N}) ds \\ & \qquad \qquad + (1+ c_0^{2N})(t\wedge t_n)^{\frac{2N}{j}-1}\int_0^{t\wedge t_n} (\|z(s)\|_N^{\frac{4N}{j}}+ \|z(s)\|_N^{2N}) ds \Big] \\ &\quad \lesssim_N \|z(0)\|_N^{2N} + \Big[\big(1+ \beta_0+ \Gamma_0 +b_0 + \gamma_0+ c_0\big)\big(1+ t\wedge t_n\big)\Big]^{2N} \Big[ t+ \int_0^t \|z({s\wedge t_n} )\|_N^{2N} ds \Big]
\end{align*}
\endgroup

By Gronwall's inequality, we find that $\|z(t\wedge t_n)\rvert_N^{2N}< g_{2N}(t\wedge t_n)$ for some monotonically increasing continuous function $g_{2N}: [0, \infty) \rightarrow [0, \infty)$. The existence of a global solution follows by the same arguments as before. 
\end{proof}

\subsection{Proof of Theorem~\ref{thm:dual_moment_formula}}\label{app:dual_moment_formula}
\begin{proof}[Proof of Theorem~\ref{thm:dual_moment_formula}]
 Recall that by the Definitions~\ref{def:transition_op} and~\ref{def:infi_generator} the process $(Z_t)_{t\geq 0}$ given by~\eqref{eq:dual_SDE} is a time-inhomogeneous polynomial process with associated transition operators $(P_{s,t})_{(s,t)\in \Delta}$ and infinitesimal generators $(\mathcal{H}_s)_{s\geq0}$. This is due to the assumption of a unique solution to the associated martingale problem.

 Recall that by the definition of polynomial processes, the transition operators and the infinitesimal generators are linear operators on the space of polynomials, and we can associate with them matrices $(\mathbf{P}_{s,t})_{(s,t)\in \Delta}$ and $(\mathbf{H}_s)_{s\geq0}$ with respect to some basis of polynomials of degree $N$, see Section~\ref{subsec:intro_poly}. In fact, it is straightforward to check that $$\mathbf{H}_t = \tilde{\mathbf{H}}(\mathbf{vec}^{\Z_0,\cc}(t)).$$

 Note that the structure of $\tilde{\mathbf{H}}(\cdot)$ is the same as $\mathbf{L}(\cdot)^{\mathsf{T}}$ introduced in Subsection~\ref{subsec:primal}. Although this was not explicit in the primal derivation, recall that we did establish in the proof of Corollary~\ref{coro:primal_moment_formula} that $\mathbf{L}^{\mathsf{T}}(\cdot)$ corresponds to the matrix-form of the infinitesimal generator of~\ref{eq:Poly_MV}. 

By the definition of time-inhomogeneous polynomial processes, we know that there exists a $\cc^*: \Delta \times \R^{N+1} \rightarrow \R^{N+1}$ such that for all $u\in \R^{N+1}$ $$ \cc^*(T,t, u) := \mathbf{P}_{t,T}\cdot u$$ 
with $$ \E[ Z_T^i \lvert \mathcal{F}_t] = \langle \bar{\Z}_t, \mathbf{P}_{t,T}\cdot e_i\rangle = \langle \bar{\Z}_t, \cc^*(T,t,e_i)\rangle $$
as well as
$$\E[\langle \bar{\Z}_T , u\rangle \lvert \mathcal{F}_t] = \sum_{i=0}^N u_i \E[\langle \bar{\Z}_T , e_i\rangle \lvert \mathcal{F}_t] =  \langle \bar{\Z}_t, \sum_{i=0}^N u_i \mathbf{P}_{t,T}\cdot e_i\rangle= 
\langle \bar{\Z}_t, \cc^*(T,t, u)\rangle.$$
Let us also recall Example~\ref{ex:P_op} for an illustration of the columns of $\mathbf{P}_{t, T}$. Even though the explicit moment formula (Theorem~\ref{thm:moment_formula}) only holds on finite time horizons (due to convergence issues of the Magnus expansion), the above relation holds on all of $\Delta$. 

By Lemma~\ref{lem:intro_fwd_bwd} it holds that 
\begin{align*}
    \frac{\d^+}{\d t} \cc^*(t,0, e_i)&=  \frac{\d^+}{\d t} \mathbf{P}_{0,t}\cdot e_i = \mathbf{P}_{0,t}\cdot \tilde{\mathbf{H}}(\mathbf{vec}^{\Z_0,\cc}(t))\cdot e_i \\ &= \mathbf{P}_{0,t}\cdot \left(\sum_{j=0}^N e_j \cdot e_j^{\mathsf{T}}\right)\cdot \tilde{\mathbf{H}}(\mathbf{vec}^{\Z_0,\cc}(t))\cdot e_i \\ & =\sum_{j=0}^N \mathbf{P}_{0,t}\cdot e_j \cdot e_j^{\mathsf{T}}\cdot \tilde{\mathbf{H}}(\mathbf{vec}^{\Z_0,\cc}(t))\cdot e_i \\& = \sum_{j=0}^N \cc^*(t, 0, e_j) \tilde{\mathbf{H}}(\mathbf{vec}^{\Z_0,\cc}(t))_{ji} \quad \textrm{ for } \cc^*(0,0,e_i)=e_i
\end{align*}
\begin{align*}
   \frac{\d^+}{\d t} \cc^*(T,t, e_i)&= \frac{\d^+}{\d t} \mathbf{P}_{t,T}\cdot e_i= -\tilde{\mathbf{H}}(\mathbf{vec}^{\Z_0,\cc}(t))\cdot \mathbf{P}_{t,T}\cdot e_i\\ & = -\tilde{\mathbf{H}}(\mathbf{vec}^{\Z_0,\cc}(t))\cdot \cc^*(T, t, e_i) \quad \textrm{ for } \cc^*(T,T,e_i)=e_i. 
\end{align*}
But since we assumed that the map $\cc$ is the unique solution to \eqref{eq:fwd_ode_c} and \eqref{eq:bwd_ode_c}, then it holds that $\cc^*=\cc$ and we proved the claim. 

\end{proof}

\subsection{Proof of Lemma~\ref{lem:fubini_type}}\label{app:fubini_type}
\begin{proof}[Proof of Lemma~\ref{lem:fubini_type}]
We recall the $L^2$-convergence of the It\^o integral as formulated e.g. in~\cite[Section 3.2, particulary Definition 2.9]{Karatzas_BM}. That is, for any square-integrable $(H_t)_{\geq 0}$ there exists a sequence of simple processes $((H^{(n)}_t)_{\geq 0})_{n\geq 0}$ such that 
$$ \lim_{n\rightarrow \infty} \E\left[ \left(\int_0^t H^{(n)}_s dW_s  - \int_0^t H_s dW_s\right)^2\right] = 0.$$

Let us also note the simple fact that for any sequence $X^{(n)} \xrightarrow[]{L^2} X$ it also holds that $\E[X^{(n)}\lvert \mathcal{M}] \xrightarrow[]{L^2} \E[X^{}\lvert \mathcal{M}]$ for any sub-$\sigma$-algebra $\mathcal{M}$ since 
\begin{align*}
\lim_{n\rightarrow \infty} \E[(\E[X^{(n)}\lvert \mathcal{M}]- \E[X^{}\lvert \mathcal{M}])^2 ]&= \lim_{n\rightarrow \infty} \E[(\E[X^{(n)}- X^{}\lvert \mathcal{M}])^2 ] \\ &\leq \lim_{n\rightarrow \infty} \E[\E[(X^{(n)}- X^{})^2\lvert \mathcal{M}] ] = \lim_{n\rightarrow \infty} \E[(X^{(n)}- X^{})^2]
\end{align*}
by Jensen's inequality and the tower property. 

We now proceed to show the first statement. Clearly, for all $n>0$ it holds that $$\E\left[ \int_0^t H^{(n)}_s dW_s \Big\lvert \Wcal_t^0 \right]=0. $$
To see this, we use the tower property, i.e. 
\begin{align*}
\E\left[ \int_0^t H^{(n)}_s dW_s \Big\lvert \Wcal_t^0 \right]&= \sum_{i=0}^{n-1} \E[H^{(n)}_{t_i}(W_{t\wedge t_{i+1}}- W_{t\wedge t_{i}})\lvert \Wcal_t^0] \\ &= 
\sum_{i=0}^{n-1} \E\left[\E[H^{(n)}_{t_i}(W_{t\wedge t_{i+1}}- W_{t\wedge t_{i}})\lvert \sigma(W_{t_i}, W^0_t)]\big\lvert \Wcal_t^0\right] \\ &= 
\sum_{i=0}^{n-1} \E\Big[H^{(n)}_{t_i}\underbrace{\E[(W_{t\wedge t_{i+1}}- W_{t\wedge t_{i}})\lvert \sigma(W_{t_i}, W^0_t)]}_{=0}\big\lvert \Wcal_t^0\Big] =0,
\end{align*}
where $\sigma(W_{t_i}, W^0_t)$ denotes the sigma algebra generated by $W_{t_i}, W^0_t$. Hence, $$\E\left[ \int_0^t H^{(n)}_s dW_s \Big\lvert \Wcal_t^0 \right] \xrightarrow[]{L^2} 0 $$
and on the other hand 
$$\E\left[ \int_0^t H^{(n)}_s dW_s \Big\lvert \Wcal_t^0 \right] \xrightarrow[]{L^2} \E\left[ \int_0^t H^{}_s dW_s \Big\lvert \Wcal_t^0 \right] $$
and it follows $ \E\left[ \int_0^t H^{}_s dW_s \big\lvert \Wcal_t^0 \right] = 0 $ almost surely. 

Similarly, we show the second statement. Note that 
\begin{align*}
\E\left[ \int_0^t H^{(n)}_s dW^0_s \Big\lvert \Wcal_t^0 \right]&= \sum_{i=0}^{n-1} \E[H^{(n)}_{t_i}(W^0_{t\wedge t_{i+1}}- W^0_{t\wedge t_{i}})\lvert \Wcal_t^0] \\ &= \sum_{i=0}^{n-1} \E[H^{(n)}_{t_i}\lvert \Wcal_t^0] (W^0_{t\wedge t_{i+1}}- W^0_{t\wedge t_{i}})= \int_0^t \E[H^{(n)}_s \lvert \Wcal^0_t] dW^0_s
\end{align*}
Moreover, for all $s,t\geq 0$ it obviously holds that $\E[H^{(n)}_s \lvert \Wcal^0_t] \xrightarrow[]{L^2}\E[H_s \lvert \Wcal^0_t] $ by assumption that $H^{(n)}  \xrightarrow[]{L^2} H$. Hence, by the $L^2$-convergence of the It\^o integral 
$$ \lim_{n\rightarrow \infty} \E\left[ \int_0^t H^{(n)}_s dW^0_s \Big\lvert \Wcal_t^0 \right] = \lim_{n\rightarrow \infty}  \int_0^t \E\left[H^{(n)}_s \big\lvert \Wcal_t^0 \right] dW^0_s \xrightarrow[]{L^2} \int_0^t \E\left[H^{}_s \big\lvert \Wcal_t^0 \right] dW^0_s.$$
On the other hand, we know that $$\E\left[\int_0^t H^{(n)}_s dW^0_s \big \lvert \Wcal_t^0 \right]\xrightarrow[]{L^2} \E\left[\int_0^t H^{}_s dW^0_s \big \lvert \Wcal_t^0 \right]$$
and thereby it follows that $$\E\left[\int_0^t H^{}_s dW^0_s \big \lvert \Wcal_t^0 \right] = \int_0^t \E\left[H^{}_s \big\lvert \Wcal_t^0 \right] dW^0_s$$ holds almost surely. 

\end{proof}

\subsection{Proof of Lemma~\ref{lem:cond_poly_squareint}}\label{app:cond_poly_squareint}
\begin{proof}[Proof of Lemma~\ref{lem:cond_poly_squareint}]
Using It\^o's formula, we obtain that 
\begin{align*}
&\E[Z_T^k \lvert \Wcal_T^0]= \E[Z_0^k \lvert \Wcal_T^0] + \int_0^T\Bigg\{\mathbf{1}_{\{k\geq2\}}\frac{k(k-1)}{2} \big(c(\E[\Z_t \lvert \Wcal_t^0]) + l^2(\E[\Z_t \lvert \Wcal_t^0])\big) \E[Z^{k-2}_t \lvert \Wcal_T^0]\\ &\,\, + \Big(k\cdot b(\E[\Z_t \lvert \Wcal_t^0]) + \frac{k(k-1)}{2}\big(\gamma(\E[\Z_t \lvert \Wcal_t^0]) + 2l(\E[\Z_t \lvert \Wcal_t^0])\Lambda(\E[\Z_t \lvert \Wcal_t^0]) \big) \Big) \E[Z^{k-1}_t \lvert \Wcal_T^0] \\ & \,\, + \big( k\cdot \beta(\E[\Z_t \lvert \Wcal_t^0]) + \frac{k(k-1)}{2}(\Gamma(\E[\Z_t \lvert \Wcal_t^0]) + \Lambda^2(\E[\Z_t \lvert \Wcal_t^0])\big) \E[Z^k_t \lvert \Wcal_T^0] \Bigg\} dt \\ & \,\,+ \E\Bigg[\int_0^T \Big( l(\E[\Z_t \lvert \Wcal_t^0])Z^{k-1}_t + \Lambda(\E[\Z_t \lvert \Wcal_t^0]) Z^{k}_t \Big) dW_t^0\Big \lvert \Wcal_T^0\Bigg]
\\ &\,\, + \E\Bigg[\int_0^T\Big(c(\E[\Z_t \lvert \Wcal_t^0])Z^{2k-2}_t + \gamma(\E[\Z_t \lvert \Wcal_t^0])Z^{2k-1}_t+ \Gamma(\E[\Z_t \lvert \Wcal_t^0]) Z^{2k}_t \Big)^{\frac{1}{2}} dW_t\Big \lvert \Wcal^0_T\Bigg].
\end{align*}
We know that the last term vanishes if the integrand is square-integrable, thanks to Lemma~\ref{lem:fubini_type}. Likewise, we can exchange integration and conditional expectation in the second-last term, if the integrand is square-integrable, again due to Lemma~\ref{lem:fubini_type}. Let us note that $\E[\Z_t \lvert \Wcal^0_T]=\E[\Z_t \lvert \Wcal^0_t]$ for all $t\leq T$. We also recall the following inequality, which we have derived in Subsection~\ref{subsec:primal} for the case of the expectation, namely that for all $i \in \{1, \dots, N\}$
\begin{align*}
\lvert \E[Z^i_t\lvert\Wcal_t^0]\rvert^{\frac{2}{i}} = \lvert \E[Z^i_t\lvert\Wcal_t^0]\rvert^{\frac{2N}{iN}} \leq \E[\lvert Z_t\rvert^{N}\lvert\Wcal_t^0]^{\frac{2}{N}}\end{align*}
and hence $$\|\E[\Z_t\lvert\Wcal_t^0]\|_N \lesssim_N \E[\lvert Z_t\rvert^N\lvert\Wcal_t^0]^{\frac{1}{N}}.$$

We now check square-integrability of the integrand of the last term. 

\begin{align*}
&\int_0^T \E\left[ \lvert c(\E[\Z_t \lvert \Wcal_t^0])Z^{2k-2}_t + \gamma(\E[\Z_t \lvert \Wcal_t^0])Z^{2k-1}_t + \Gamma(\E[\Z_t \lvert \Wcal_t^0])Z^{2k}_t \rvert \right]dt \\ &\quad \lesssim_N \int_0^T c_0\left (\E[\lvert Z_t\rvert^{2k-2}]+ \E\left[\big\lvert\E[\lvert Z_t\rvert^N\lvert \Wcal_t^0]^{\frac{2}{N}}Z_t^{2k-2}\big\rvert\right]\right) dt \\ & \qquad + \int_0^T \gamma_0\left (\E[\lvert Z_t\rvert^{2k-1}]+ \E\left[\big\lvert\E[\lvert Z_t\rvert^N\lvert \Wcal_t^0]^{\frac{1}{N}}Z_t^{2k-1}\big\rvert\right]\right) + \Gamma_0 \E[\lvert Z^{2k}_t\rvert] dt 
\end{align*}
Note that for $k> 1$ 
\begin{align*}
&\E\left[\big\lvert\E[\lvert Z_t\rvert^N\lvert \Wcal_t^0]^{\frac{2}{N}}Z_t^{2k-2}\big\rvert\right] = \E\left[\E[\lvert Z_t\rvert^N\lvert \Wcal_t^0]^{\frac{2}{N}}\E[\lvert Z_t\rvert^{2k-2}\lvert \Wcal_t^0]\right] \\ &\leq \E\left[\E[\lvert Z_t\rvert^{2N}\lvert \Wcal_t^0]^{\frac{1}{N}}\E[\lvert Z_t\rvert^{2(k-1)}\lvert \Wcal_t^0]^{\frac{(k-1)N}{(k-1)N}}\right] \leq \E\left[\E[\lvert Z_t\rvert^{2N}\lvert \Wcal_t^0]^{\frac{1}{N}}\E[\lvert Z_t\rvert^{2N}\lvert \Wcal_t^0]^{\frac{k-1}{N}}\right] \\ & = 
\E\left[\E[\lvert Z_t\rvert^{2N}\lvert \Wcal_t^0]^{\frac{k}{N}}\right] \leq (1 + \E\left[\E[\lvert Z_t\rvert^{2N}\lvert \Wcal_t^0]\right]) = \left(1 + \E[ \lvert Z_t\rvert^{2N}]\right),
\end{align*}
while for $k=1$
$$\E\left[\big \lvert  \E[\lvert Z_t\rvert^N\lvert \Wcal_t^0]^{\frac{2}{N}}Z_t^{2k-2} \big\rvert  \right] = \E\left[  \E[\lvert Z_t\rvert^N\lvert \Wcal_t^0]^{\frac{2}{N}}  \right] \leq \E\left[1+  \E[\lvert Z_t\rvert^{2N}\lvert \Wcal_t^0]\right] = \left(1+ \E[\lvert Z_t\rvert^{2N}]\right) $$
and likewise for $k\geq 1$: 
\begin{align*}
&\E\left[\big\lvert\E[\lvert Z_t\rvert^N\lvert \Wcal_t^0]^{\frac{1}{N}}Z_t^{2k-1}\big\rvert\right] = \E\left[\E[\lvert Z_t\rvert^N\lvert \Wcal_t^0]^{\frac{1}{N}}\E[\lvert Z_t\rvert^{2k-1}\lvert \Wcal_t^0]\right] \\ &\leq \E\left[\E[\lvert Z_t\rvert^{N}\lvert \Wcal_t^0]^{\frac{1}{N}}\E[\lvert Z_t\rvert^{2k-1)}\lvert \Wcal_t^0]^{\frac{(2k-1)N}{(2k-1)N}}\right] \leq \E\left[\E[\lvert Z_t\rvert^{N}\lvert \Wcal_t^0]^{\frac{1}{N}}\E[\lvert Z_t\rvert^{N}\lvert \Wcal_t^0]^{\frac{2k-1}{N}}\right] \\ & = 
\E\left[\E[\lvert Z_t\rvert^{N}\lvert \Wcal_t^0]^{\frac{2k}{N}}\right] \leq (1 + \E\left[\E[\lvert Z_t\rvert^{2N}\lvert \Wcal_t^0]\right]) = \left(1 + \E[ \lvert Z_t\rvert^{2N}]\right).
\end{align*}
Similarly, it holds for the second last term that 
\begin{align*}
&\int_0^T \E\left[\big\lvert l(\E[\Z_t \lvert \Wcal_t^0])Z^{k-1}_t + \Lambda(\E[\Z_t \lvert \Wcal_t^0]) Z^{k}_t  \big\rvert^2\right] dt \\ & \qquad \leq \int_0^T l_0\left(\E[\lvert Z_t\rvert^{2k-2}]+ \E\left[ \big \lvert \E[\lvert Z_t\rvert^N\lvert \Wcal_t^0]^{\frac{2}{N}} Z^{2k-2}_t \big \rvert\right]\right) + \Lambda_0 \E[Z_t^{2k}]  dt
\end{align*}
which can be treated using the same inequalities as above. 

Hence, it is enough to show that $\sup_{t\in [0,T]} \E[ \lvert Z_t\rvert^{2N}]< \infty$ for all $T\in [0, \infty)$ in order to show square-integrability.  To prove this, we can proceed exactly like in the proof of Lemma~\ref{lem:new_ass_maps} and show the statement using a Gronwall-type inequality. 
Again, we introduce a sequence of stopping-times $(\tau_n)_{n\in \N}$ with $\tau_n= \inf\{t\geq 0 \lvert \, \lvert Z_t\rvert \geq n\}$. We obtain that
\begin{align*}
&\E\left[\lvert Z_{t\wedge \tau_n}\rvert^{2N}\right] = \E\Bigg[ \Big\lvert Z_0 + \int_0^{t\wedge \tau_n} b(\E[\Z_s \lvert \Wcal_s^0]) + \beta(\E[\Z_s \lvert \Wcal_s^0])Z_s ds \\ & \qquad +\int_0^{t\wedge \tau_n} \sqrt{c(\E[\Z_s \lvert \Wcal_s^0]) + \gamma(\E[\Z_s \lvert \Wcal_s^0]) Z_s + \Gamma(\E[\Z_s \lvert \Wcal_s^0]) Z^2_s} dW_s \\ & \qquad +\int_0^{t\wedge \tau_n} l(\E[\Z_s \lvert \Wcal_s^0])+\Lambda(\E[\Z_s \lvert \Wcal_s^0])Z_sdW^0_s \Big\rvert^{2N}\Bigg]\\ & \lesssim_N \E\left[\lvert Z_0\rvert^{2N}\right] + \E\Bigg[t^{2N-1}\int_0^{t} \big \lvert b(\E[\Z_{s\wedge \tau_n} \lvert \Wcal_{s}^0]) + \beta(\E[\Z_{s\wedge \tau_n} \lvert \Wcal_{s}^0]) Z_{s\wedge \tau_n} \big \rvert^{2N}ds \Bigg] \\ & \qquad + \E\Bigg[\int_0^t \big\lvert c(\E[\Z_{s\wedge \tau_n} \lvert \Wcal_{s}^0])+ \gamma(\E[\Z_{s\wedge \tau_n} \lvert \Wcal_{s}^0]) Z_{s\wedge \tau_n} + \Gamma(\E[\Z_{s\wedge \tau_n} \lvert \Wcal_{s}^0])Z^2_{s\wedge \tau_n} \big\rvert^{N} ds\Bigg] \\ & \qquad + \E\Bigg[\int_0^t \big\lvert l(\E[\Z_{s\wedge \tau_n} \lvert \Wcal_s^0])+ \Lambda(\E[\Z_{s\wedge \tau_n} \lvert \Wcal_s^0])Z_{s\wedge \tau_n}\big\rvert^{2N} ds\Bigg] \\ & \lesssim_N \E\left[\lvert Z_0\rvert^{2N}\right] + t^{2N-1}\int_0^{t}  b_0^{2N}\left(1+ \E\left[\E[\lvert Z_{s\wedge \tau_n}\rvert^N \lvert \Wcal_s^0]^2\right]\right) + \beta_0^{2N} \E[\lvert Z_{s\wedge \tau_n} \rvert^{2N}]ds \\ & \quad + \int_0^t c_0^{N}\left(1+ \E\left[\E[\lvert Z_{s\wedge \tau_n}\rvert^N \lvert \Wcal_s^0]^2\right]\right) + \Gamma_0^{N}\E[\lvert Z_{s\wedge\tau_n}\rvert^{2N}] ds \\ & \quad + \int_0^t \gamma_0^N \left( \E[\lvert Z_{s\wedge \tau_n}\rvert^N] + \E\left[\E[\lvert Z_{s\wedge \tau_n}\rvert^{N}\lvert \Wcal_s^0]\lvert Z_{s\wedge\tau_n}\rvert^N\right]\right) ds \\& \quad +  \int_0^t l_0^{2N}\left(1+ \E\left[\E[\lvert Z_{s\wedge \tau_n}\rvert^N \lvert \Wcal_s^0]^2\right]\right) + \Lambda_0^{2N} \E[\lvert Z_{s\wedge \tau_n} \rvert^{2N}] ds \\ & \lesssim_N \E\left[\lvert Z_0\rvert^{2N}\right] + (b_0t)^{2N} + c_0^{2N}t+ \gamma_0^{2N}t + l_0^{2N}t \\ & \quad + \left(b_0^{2N}t^{2N-1} + \beta_0^{2N}t^{2N-1} + c_0^N + \gamma_0^N+ \Gamma_0^N+ l_0^{2N}+ \Lambda_0^{2N}\right)\int_0^t \E[\lvert Z_{s\wedge\tau_n}\rvert^{2N}] ds
\end{align*}
where we used that $$\E\left[\E[\lvert Z_{s\wedge \tau_n}\rvert^N \lvert \Wcal_s^0]^2\right] \leq \E\left[\E[\lvert Z_{s\wedge \tau_n}\rvert^{2N} \lvert \Wcal_s^0]\right] = \E\left[\lvert Z_{s\wedge \tau_n}\rvert^{2N} \right] $$ and $$\E\left[\E[\lvert Z_{s\wedge \tau_n}\rvert^{N}\lvert \Wcal_s^0]\lvert Z_{s\wedge\tau_n}\rvert^N\right]= \E\left[\E[\lvert Z_{s\wedge \tau_n}\rvert^{N}\lvert \Wcal_s^0]^2\right]\leq \E\left[\lvert Z_{s\wedge \tau_n}\rvert^{2N} \right].$$

Hence, by an application of Gronwall's inequality, we have indeed found an upper bound for $\E\left[\lvert Z_{t\wedge \tau_n}\rvert^{2N}\right] \leq \Theta(t)$ for a continuous increasing function $\Theta: [0, \infty)\rightarrow [0, \infty)$. And moreover, by the fact that $(\tau_n)_{n\in\N}$ and Fatou's lemma, we obtain that also $\E\left[\lvert Z_{t}\rvert^{2N}\right] \leq \Theta(t)$. 
\end{proof}

\subsection{Proof of Theorem~\ref{thm:cond_existence_MP}}\label{app:cond_existence_MP}
\begin{proof}[Proof of Theorem~\ref{thm:cond_existence_MP}]
To see this, we describe $(Z_t, \E[\Z_t \lvert \Wcal_t^0])_{t\geq 0}$ as a joint system of SDEs  and recall \cite[Theorem 3.3.10]{EK} regarding sufficient conditions for the existence of solutions to SDEs via solutions to the martingale problem. 
To this end, let us write $X_t= (Z_t, \E[\Z_t \lvert \Wcal_t^0])$ for all $t\geq0$ which is described by:
$$ dX_t = B(X_t) dt + \Sigma(X_t) d\tilde{W}_t
$$
where $B: \R^{N+1} \rightarrow \R^{N+1}$, $\Sigma: \R^{N+1} \rightarrow \R^{(N+1)\times 2}$ and $\tilde{W}_t= (W_t, W_t^0)$ for all $t\geq 0$. In particular, we emphasize that we are now in a time-homogeneous setting. 
By \cite[Theorem 3.3.10]{EK} we are left to check that there exists a $K\geq 0$, such that $$ (X_t)_j \cdot B(X_t)_j \leq K(1+ \|X_t\|^2)\textrm{ and } \left\lvert \left(\Sigma(X_t)\Sigma(X_t)^{\mathsf{T}}\right)_{ij}\right\rvert \leq K(1+ \|X_t\|^2)$$
for all $0\leq i, j \leq N$. Note that we use the notation $(X_t)_0:= Z_t$ and $(X_t)_k:= \E[Z^k_t \lvert \Wcal_t^0]$ for all $1\leq k \leq N$. We assume w.l.o.g that $N$ is even\footnote{If the minimal $N$ we consider would actually be odd, it is no problem to go up to $N+1$ instead.}.  Note, that 
\begin{align*}
B(X_t)_0 &= b(\E[\Z_t\lvert \Wcal_t^0]) + \beta(\E[\Z_t\lvert \Wcal_t^0])Z_t\\
B(X_t)_k &= \Big(k\cdot b(\E[\Z_t \lvert \Wcal_t^0]) + \frac{k(k-1)}{2}\big(\gamma(\E[\Z_t \lvert \Wcal_t^0]) \\ & \qquad+ l(\E[\Z_t \lvert \Wcal_t^0])\Lambda(\E[\Z_t \lvert \Wcal_t^0]) \big) \Big) \E[Z^{k-1}_t \lvert \Wcal_t^0]\\ &\qquad  +\mathbf{1}_{\{k\geq 2\}} \frac{k(k-1)}{2} \big(c(\E[\Z_t \lvert \Wcal_t^0]) + l^2(\E[\Z_t \lvert \Wcal_t^0])\big) \E[Z^{k-2}_t \lvert \Wcal_t^0] \\ & \qquad + \big( k\cdot \beta(\E[\Z_t \lvert \Wcal_t^0]) + \frac{k(k-1)}{2}(\Gamma(\E[\Z_t \lvert \Wcal_t^0]) + \Lambda^2(\E[\Z_t \lvert \Wcal_t^0])\big) \E[Z^k_t \lvert \Wcal_t^0] 
\end{align*}

Clearly it holds that $$(X_t)_0 \cdot B(X_t)_0 \lesssim_N Z_t\cdot \left(b_0(1+ 
\E[\lvert Z_t\rvert^N\lvert \Wcal_t^0]^{\frac{1}{N}})+ \beta_0Z_t\right)\leq K(1+ \|X_t\|^2)$$
since $ ab \leq (a+b)^2 \lesssim \lvert a\rvert^2 + \lvert b\rvert^2 $ for any $a,b \in \R$. Similarly, it holds for $k \in \{1, \dots, N\}$
\begin{align*}
(X_t)_k \cdot B(X_t)_k \leq & \Big(kb_0 + \frac{k(k-1)}{2}(\gamma_0 + l_0\Lambda_0)\Big)\Big\lvert \E[Z_t^{k-1}\lvert \Wcal_t^0] \E[Z_t^{k}\lvert \Wcal_t^0]\Big\rvert\Big(1+ \|\E[\Z_t\lvert \Wcal^0_t]\|_N\Big) \\ & + \mathbf{1}_{\{k\geq 2\}} \frac{k(k-1)}{2} \Big( c_0+l_0^2\big(1+ \|\E[\Z_t\lvert \Wcal_t^0]\|_N\big)^2\Big) \Big\lvert \E[Z^{k-2}_t \lvert \Wcal_t^0] \E[Z_t^k\lvert \Wcal_t^0] \Big \rvert \\& + \Big(k\beta_0 + \frac{k(k-1)}{2}(\Gamma_0 + \Lambda_0^2)\Big)\Big\lvert \E[Z_t^k\lvert \Wcal_t^0] \Big\rvert^2
\\ \leq& K(1 + \|X_t\|^2)
\end{align*}
where we used that $\|\E[\Z_t\lvert\Wcal_t^0]\|_N \lesssim_N \E[\lvert Z_t\rvert^N \lvert \Wcal_t^0]^{\frac{1}{N}}$ together with 
\begin{align*}
\Big \lvert \E[Z_t^{k-1}\lvert \Wcal_t^0] \E[Z_t^{k}\lvert \Wcal_t^0] \Big\rvert &= \big \lvert \E[Z_t^{k-1}\lvert \Wcal_t^0]\big\rvert^{\frac{N(k-1)}{N(k-1)}}\big\lvert\E[Z_t^{k}\lvert \Wcal_t^0] \big\rvert^{\frac{Nk}{Nk}} \\ &\leq \E[\lvert Z_t\rvert^{N}\lvert \Wcal_t^0]\big\rvert^{\frac{(k-1)}{N}} \E[\lvert Z_t\rvert^{N}\lvert \Wcal_t^0]\big\rvert^{\frac{k}{N}} \\ &= \E[\lvert Z_t\rvert^{N}\lvert \Wcal_t^0]\big\rvert^{\frac{2k-1}{N}} 
\end{align*}
and similarly 
\begin{align*}
\Big \lvert \E[Z_t^{k-2}\lvert \Wcal_t^0] \E[Z_t^{k}\lvert \Wcal_t^0] \Big\rvert \leq \E[\lvert Z_t\rvert^{N}\lvert \Wcal_t^0]\big\rvert^{\frac{2k-2}{N}}
\end{align*}
and hence 
\begin{align*}
\big \lvert \E[Z_t^{k-1}\lvert \Wcal_t^0] \E[Z_t^{k}\lvert \Wcal_t^0] \big\rvert \|\E[\Z\lvert\Wcal_t^0]\|_N &\lesssim \E[\lvert Z_t\rvert^{N}\lvert \Wcal_t^0]\big\rvert^{\frac{2k}{N}} \leq (1 +  \E[\lvert Z_t\rvert^{N}\lvert \Wcal_t^0]\big\rvert^{2})\\
\big \lvert \E[Z_t^{k-2}\lvert \Wcal_t^0] \E[Z_t^{k}\lvert \Wcal_t^0] \big\rvert \|\E[\Z\lvert\Wcal_t^0]\|^2_N &\lesssim \E[\lvert Z_t\rvert^{N}\lvert \Wcal_t^0]\big\rvert^{\frac{2k}{N}} \leq (1 +  \E[\lvert Z_t\rvert^{N}\lvert \Wcal_t^0]\big\rvert^{2}).
\end{align*}
On the other hand, for $j, k \in \{1, \dots, N\}$ we find 
\begin{align*}
\left(\Sigma(X_t)\Sigma(X_t)^{\mathsf{T}}\right)_{jk}&= \underbrace{\Sigma(X_t)_{j0}\Sigma(X_t)_{k0}}_{=0}+\Sigma(X_t)_{j1}\Sigma(X_t)_{k1}\\ &= \Big(l(\E[\Z_t \lvert \Wcal_t^0])\E[Z^{k-1}_t\lvert \Wcal_t^0] + \Lambda(\E[\Z_t \lvert \Wcal_t^0])\E[ Z^{k}_t \lvert \Wcal_t^0]\Big)\\ & \quad \cdot \Big(l(\E[\Z_t \lvert \Wcal_t^0])\E[Z^{j-1}_t\lvert \Wcal_t^0] + \Lambda(\E[\Z_t \lvert \Wcal_t^0]) \E[Z^{j}_t\lvert \Wcal_t^0] \Big).
\end{align*}
Using the growth conditions of $l, \Lambda$ and the fact that
\begin{align}\label{eq:ineq_cond} \|\E[ \Z_t \lvert \Wcal_t^0]\|_N \E[\lvert Z_t\rvert^{k-1} \lvert \Wcal_t^0]\lesssim_N \E\left[\lvert Z_t\rvert^N  \lvert \Wcal_t^0\right]^{\frac{k}{N}}\leq 1+ \E\left[ \lvert Z_t\rvert^N \big \lvert \Wcal_t^0 \right]
\end{align}
it is clear that 
$$\left\lvert \left(\Sigma(X_t)\Sigma(X_t)^{\mathsf{T}}\right)_{jk}\right\rvert \leq K(1+ \|X_t\|^2)$$
for all $i,j\in\{1, \dots, N\}$. Whereas 
\begin{align*}
\left\lvert\left(\Sigma(X_t)\Sigma(X_t)^{\mathsf{T}}\right)_{00}\right\rvert&= \left\lvert\left(\Sigma(X_t)_{00}\right)^2+ \left(\Sigma(X_t)_{01}\right)^2 \right\rvert\\ &=  \Big\lvert c(\E[\Z_t \lvert \Wcal_t^0]) + \gamma(\E[\Z_t \lvert \Wcal_t^0])Z_t+ \Gamma(\E[\Z_t \lvert \Wcal_t^0])Z_t^2 \\ & \quad + l(\E[\Z_t \lvert \Wcal_t^0])^2 + 2 l(\E[\Z_t \lvert \Wcal_t^0])\Lambda(\E[\Z_t \lvert \Wcal_t^0])Z_t+ \Lambda(\E[\Z_t \lvert \Wcal_t^0])^2Z_t^2\Big\rvert\\ &\leq K(1 + \|X_t\|^2)
\end{align*}
by assumptions on the maps $c, l, \gamma,  \Gamma, \Lambda$. We are now only left with the terms $\left(\Sigma(X_t)\Sigma(X_t)^{\mathsf{T}}\right)_{0k}$ for $ k \in \{1, \dots, N\}$, where 
\begin{align*}
\left\lvert \left(\Sigma(X_t)\Sigma(X_t)^{\mathsf{T}}\right)_{0k}\right\rvert &= \big\lvert \Sigma(X_t)_{01}\Sigma(X_t)_{k1}\big\rvert = \Big\lvert \Big(l(\E[\Z_t \lvert \Wcal_t^0]) +  \Lambda(\E[\Z_t \lvert \Wcal_t^0])Z_t\Big) \\& \quad \cdot \left(l(\E[\Z_t \lvert \Wcal_t^0])\E[Z^{k-1}_t \lvert \Wcal_t^0] + \Lambda(\E[\Z_t \lvert \Wcal_t^0]) \E[Z^{k}_t\lvert \Wcal_t^0] \right)\Big \rvert \\ &\lesssim_N \Big\lvert \Big(l_0 + l_0\E[\lvert Z_t\rvert^N \lvert \Wcal_t^0]^{\frac{1}{N}} + \Lambda_0\lvert Z_t\rvert \Big) \\ & \quad \cdot  \left(l_0\E[ \lvert Z_t \rvert^N\lvert \Wcal_t^0]^{\frac{k-1}{N}} + l_0\E[ \lvert Z_t \rvert^N\lvert \Wcal_t^0]^{\frac{k}{N}}+ \Lambda_0\E[ \lvert Z_t\rvert^k\lvert \Wcal_t^0]  \right) \Big \rvert \\& \leq K(1+\|X_t\|^2).
\end{align*}
\end{proof}

\end{document}